\documentclass[11pt]{amsart}
\usepackage{amsmath,amsthm,amssymb}
\usepackage{color}
\usepackage{epsfig}
\usepackage{graphicx}

\def\E{{\mathbb{E}}}

\def\P{{\mathbb{P}}}

\def\R{{\mathbb{R}}}

\def\Z{{\mathbb{Z}}}
\def\PP{{\mathbb{P}}}

\newcommand{\reals}{{\mathbb R}}
\newcommand{\integer}{{\mathbb Z}}

\newcommand{\pintegers}{{\integer_+}}

\newcommand{\wt}{\widetilde}
\newcommand{\s}{\sigma}

\def\8{\infty}

\def\E{\mathbb{E}}
\def\P{\mathbb{P}}

\def\<{\langle}
\def\>{\rangle}

\renewcommand{\d}{\delta}
\renewcommand{\a}{\alpha}
\renewcommand{\b}{\beta}

\renewcommand{\l}{\lambda}
\renewcommand{\L}{\Lambda}

\newcommand{\e}{{\mathrm{e}}}
\newcommand{\eps}{\varepsilon}

\newcommand{\g}{\gamma}

\newcommand{\ov}{\overline}

\newtheorem{thm}[equation]{Theorem}
\newtheorem{rem}[equation]{Remark}

\newtheorem{lem}[equation]{Lemma}

\newtheorem{prop}[equation]{Proposition}
\theoremstyle{definition}

\numberwithin{equation}{section}

\begin{document}
%{\color{red} \it Wstep trzeba napisc od nowa. Ponizsze jest skopiowane z %Collamora. Trzeba tez dolozyc literature.}

\title{Pointwise estimates for first passage times of perpetuity sequences}
\author[D. Buraczewski, E. Damek  and J. Zienkiewicz]
{D.~Buraczewski, E.~Damek,  and J.~Zienkiewicz}
\address{D. Buraczewski, E. Damek, and J. Zienkiewicz\\ Instytut Matematyczny\\ Uniwersytet Wroclawski\\ 50-384 Wroclaw\\
pl. Grunwaldzki 2/4\\ Poland}
\email{dbura@math.uni.wroc.pl\\ edamek@math.uni.wroc.pl\\
zenek@math.uni.wroc.pl}

\begin{abstract}
We consider first passage times $\tau_u = \inf\{n:\; Y_n>u\}$ for the perpetuity sequence
$$
Y_n = B_1 + A_1 B_2 + \cdots + (A_1\ldots A_{n-1})B_n,
$$ where $(A_n,B_n)$ are i.i.d.  random variables with values in $\R ^+\times \R$. 
Recently, a number of limit theorems related to $\tau_u$ were proved including the law of large numbers, the central limit theorem and large deviations theorems (see \cite{BCDZ}). 
We obtain a precise asymptotics of the sequence $\P[\tau_u =  \log u/\rho ]$, $\rho >0$, $u\to \infty $ which considerably improves the previous results of \cite{BCDZ}. There, probabilities $\P[\tau_u \in I_u]$ were identified, for some large intervals $I_u$ around $k_u$, with lengths growing at least as $\log\log u$. Remarkable analogies and differences to random walks \cite{BM,Laley} are discussed.
\end{abstract}

\thanks{The research was  supported by the NCN grant DEC-2012/05/B/ST1/00692.}% The second and third authors  were partially %supported by the NCN grant 2014/15/B/ST1/00060}


 \maketitle
\section{Introduction}

\subsection{The perpetuity sequence and first passage times} Let $(A_n,B_n)$ be i.i.d.~(independent identically distributed) random variables with values in $\R ^+\times \R$.
We consider the perpetuity sequence
\begin{equation*}
%\label{intro1}
Y_n = B_1 + A_1 B_2 + \cdots + (A_1 \cdots A_{n-1}) B_n, \quad n=1,2,\ldots
\end{equation*}
If $\E [ \log A] <0$ and $\E[\log^+|B|]<\8$,
 $Y_n$ converges a.s. to the random
variable
\begin{equation*}
%\label{eq: formula R}
Y = \sum_{n=0}^\8 A_1\ldots A_{n}B_{n+1},
\end{equation*}
that is the unique solution to the random
difference equation
\begin{equation*}
%\label{dif recurrence}
Y \stackrel{d}{=} AY+B, \qquad \mbox{$Y$ independent
of }(A,B).
\end{equation*}
%If we assume additionally that there is $\a _0$ such that $\E A^{\a_0}=1$, then
%\begin{equation*}
%%\label{eq: kesten}
%\lim_{u\to \8}u^{\a _0} \P[Y>u] = c_+, \quad \lim_{u\to \8}u^{\a _0}\P[Y<-u]=c_-,
%\end{equation*}
%for some constants $c_-,c_+$ satisfying $c_- +c_+>0$.

%\eqref{eq: kesten}, proved in the fundamental works Kesten  \cite{Kesten} and Goldie \cite{Go}, has numerous consequences both for

 The perpetuity process $\{Y_n\}$ is frequently present in both applied  and theoretical problems. On one hand, the perpetuity sequence plays an important role in analyzing the ARCH and GARCH financial time series models, see
 Engle \cite{Engle} and Mikosch \cite{Mikosch}. On the other, it is connected to the random walk in random environment \cite{KKS}, the weighted branching process and the branching random walk, see Guivarc'h \cite{Guivarc'h}, Liu \cite{Liu} and Buraczewski \cite{Buraczewski}. We refer the reader to  recent monographs \cite{BDM,Iksanov:2017} for an overview on the subject.
%of theoretical properties of perpetuities and their applications.

%\subsection{Exceedence times}

The main objective of this paper is to study the asymptotic behavior of the  first passage time
$$
\tau _u:= \inf \left\{ n:  Y_n > u \right\}
%=\inf\{n: M_n > u\}
\quad \mbox{as} \:\: u \to \infty.
$$
This is a basic question motivated partly by similar problems considered for random walks $\log \Pi_n=\log(  A_1\ldots A_n)$, see the work of Lalley \cite{Laley}. Some partial results in this direction were proved by Nyrhinen \cite{N1, N2}.
  An essential progress has been  recently achieved in \cite{BCDZ} and the aim of the present paper is to pursue the investigation further.

The tail of $Y$ was analysed under the Cram\'er condition
\begin{equation}\label{eq:cramer}
  \E\big[ A^{\a_0}\big] = 1 \ \mbox{ for some }\a_0>0.
\end{equation}
Then Kesten \cite{Kesten} and Goldie \cite{Go} proved that\footnote{We  write $f(u) \sim g(u)$ for two functions $f$ and $g$, if $f(u)/g(u) \to 1$ as $u\to\8$.}
\begin{equation}
\label{eq: twostarrrr}
\P[Y>u] \sim c_+ u^{-\a_0} \quad \mbox{as } u\to\8,
\end{equation}
which entails
\begin{equation}
\label{eq: twostar}
\P[\tau_u <\8] \sim c_0 u^{-\a_0} \quad \mbox{as } u\to\8.
\end{equation}

In \cite{BCDZ} the authors proved the law of large numbers (\cite{BCDZ}, Lemma 2.1)
$$
\frac{\tau_u}{\log u}\; \bigg| \tau_u <\8 \stackrel{\P}{\longrightarrow} \frac 1{\rho_0}
$$
and the central limit theorem (\cite{BCDZ}, Theorem 2.2)
\begin{equation}\label{eq:clt}
\frac{\tau_u - \log u/\rho_0}{ \sigma_0 \rho_0^{-3/2} \sqrt{  \log u}}\; \bigg| \tau_u <\8 \stackrel{d}{\longrightarrow} N(0,1),
\end{equation}
where
$\rho_0 = \E[A^{\alpha_0} \log A]$, $\sigma_0 = \E[ A^{\alpha_0}(\log A )^2]$,  and $\stackrel{\PP}\longrightarrow$ (resp. $\stackrel{d}\longrightarrow$) denotes convergence in probability (resp. in ditribution).

They also considered large deviations and showed that for $\rho>\rho_0$ (\cite{BCDZ}, Theorem 2.1)
$$
\P\bigg(   \frac{\tau_u}{\log u} < \frac 1{\rho}    \bigg) \sim \frac{c(\rho)}{\sqrt{\log u}} u^{-I(\rho)}
$$
for some rate function $I(\rho)$ that will be defined below (see \eqref{eq:onestar}). Indeed, the result in \cite{BCDZ}  is  more precise and describes the asymptotic of
$$
\P[\tau_u\in I_u]
$$
for intervals $I_u$ around $ \log u/\rho$ of the length of the order $\log\log u$.
 This implies that $\log u$ is the correct scaling for perpetuities as it is for the random walks. More details will be given below.

 In this paper we go a step further and, under some continuity assumption on $A$, we describe the pointwise behavior of $\tau_u$, that is the asymptotic of
$$
\P\big[\tau_u = \lfloor \log u /\rho \rfloor  \big], \quad u\to \infty
$$ for any $\rho > 0$.

The results we obtain are partly with analogy to random walks but partly they are completely different and reveal essential differences in behavior of perpetuities and the corresponding random walk, see Theorems \ref{thm: main theorem} and \ref{mthm2}.

\subsection{Main results}

While studying perpetuities the main role is played by fluctuations of the random walk $\Pi_n = A_1\ldots A_n$; see \cite{BDM,Go}. In this paper  the contractive case $\E \log A < 0$ is studied, which entails that $\Pi_n$ converges to 0 a.s. On the other hand our hypotheses imply that $\P(A>1)>0$, thus the process $\Pi_n$ may attain (with small probability) arbitrary large values.

Properties of both processes $\{\Pi_n\}$ and $\{Y_n\}$ are essentially determined by the moments generating functions
$$\lambda(s) = \E[A^s], \qquad \mbox {and} \qquad \Lambda(s) = \log\lambda(s).$$
We denote $\a_\8 = \sup \{\a:\; \lambda(\a)<\8\}$ and $\a_{\min}$ ={\rm argmin}$\lambda(\a)$.
Then both functions $\lambda$ and $\Lambda$ are smooth and convex on their domain $[0,\a_{\8})$.

 Recall that the convex conjugate (or the Fenchel-Legendre transform) of $\Lambda$ is defined by
$$
\Lambda^*(x) = \sup_{s\in \R}\{sx - \Lambda(s)\}, \quad x\in\R.
$$ $\Lambda^*$ appears in studying large deviations problems for random walks. Its various properties can be found in Dembo, Zeitouni \cite{DZ}. Given $\a$ and $\rho=\Lambda'(\a)$ we consider
$
\overline \a = \Lambda^*(\rho)/\rho.
$ An easy calculation shows that
\begin{equation}\label{eq:onestar}
\overline{\alpha} = \alpha - \frac{\Lambda(\alpha)}{\Lambda'(\alpha)}.
\end{equation}
The parameter $\overline \a$ arises in the classical large deviations theory for random walks.
As we will see below, $\overline \a$  plays a  crucial role in our results.
 This parameter has a geometric interpretation: the tangent line to $\Lambda$ at point $(\alpha,\Lambda(\alpha))$ intersects the
$x$-axis at $(\overline{\alpha},0)$. See the figure below.

%\begin{center}
%  \includegraphics[scale=0.45]{lam.eps}
%\end{center}

Our main result is the following

\begin{thm}
\label{thm: main theorem}
Given $\rho>0$ suppose there exists $\a < \a_\8$ such that $\rho = \Lambda'(\a)$. Assume additionally
\begin{eqnarray}
\label{H1} &&\E \log A <0;\\
\label{H2} &&\mbox{$\E A^{\a+\eps}<\8$ and  $\E |B|^{\a+\eps}<\8$ for some }\eps>0;\\
\label{H1.4} &&\mbox{either $\a_{\min}\le 1$ or $\Lambda(1) < \Lambda(\a)$;}\\
\label{H1.3} &&\mbox{there are $(a_1,b_1), (a_2,b_2)\in {\rm supp\, } \mu$ such that $a_1<1$, $a_2>1$ and $\frac{b_2}{1-a_2}< \frac{b_1}{1-a_1}$};\\
%\label{H1.4} &&\mbox{if $k_u>n_0$ then let $\a$ be such that $k_u=\frac{\log %u}{\Lambda'(\a)}$, then $\a<1$ or $\Lambda(1)<\Lambda(\a)$;}\\
\label{H1.6} &&\mbox{the law of } A\ \mbox{has density }  f_A(a)\ \mbox{satisfying}\ f_A(a)\!\leq C(1+a) ^{-D}  \mbox{for some}\    D\!>\!1\!+\!\a ;\\
\label{H1.2} && \P[A\in da, B\in db] \le f_A(a)da d\nu(b) \mbox{ for some probability measure $\nu$}.
\end{eqnarray}
%{\color{red}   Zly warunek  (na wyczucie, nie spr, ), wystarczy gestosc z odpowiednim oszacowaniem typu produktowego}
Then
\begin{equation}
\label{eq: aymp}
\P\big[\tau_u=\lfloor \log u/\rho\rfloor \big] \sim \frac{c(\a) \lambda(\a)^{-\Theta(u)}}{\sqrt{\log u} } \, u^{-\ov \a},\qquad \mbox{ as } u\to\8,
\end{equation}
for some strictly positive constant $C$ and $\Theta(u) = k_u - \lfloor k_u\rfloor $,
where $k_u = \log u/ \rho$.
\end{thm}

 The next result  shows that
 assumption \eqref{H1.4} in Theorem \ref{thm: main theorem} is indispensable and if
%It turns out that if \eqref{H1.4} is not satisfied, then the  analogy with path %behaviour of random walk breaks down.
$\rho = \Lambda'(\a)$ is close to 0, the behavior of the passage time may be of different order.
\begin{thm}
\label{mthm2}
Given $\rho>0$ suppose there exists $\a < \a_\8$ such that $\rho = \Lambda'(\a)$ and hypothesis \eqref{H1}, \eqref{H2} are satisfied. Assume additionally
\begin{eqnarray}
%\label{H1} &&\E \log A <0;\\
%\label{H2} &&\mbox{$\E A^{\a+\eps}<\8$ and  $\E |B|^{\a+\eps}<\8$ for some }\eps>0;\quad\quad\quad\quad\quad\quad\\
\label{H2.5} && \Lambda(\a)<\Lambda(1);\\
\label{H2.1} && \mbox{$B>0$, a.s.;}
\end{eqnarray}
there are $0<b_1 < b_2$, a measure $\nu $ and a non vanishing continuous function $g_A$ such that
\begin{equation}
\label{H2.3} \P[A\in da, B_n\in {\bf 1}_{(b_1,b_2)}(b) db ] \ge g_A(a)da d\nu(b).
\end{equation}
Then there exists $\d>0$ such that
$$
\P\big[\tau_u=\lfloor \log u /\rho\rfloor \big]
%\P[T_u = \Lambda '(\a )^{-1}] = \P[\tau=k_u]
 \ge  \frac{C}{\sqrt{\log u} } \, u^{-\ov \a}\, u^\d.
$$
\end{thm}

\subsection{Some comments}

Study of $\tau _u$ is partly motivated
%the probability of the first passage time for perpetuities was
by the work of Lalley \cite{Laley}, who considered this problem for the negatively drifted  random walk $S_n = \log \Pi_n$. Let $\tau'_u = \inf\{n: S_n > \log u\}$.
Lalley proved the central limit theorem (as in \eqref{eq:clt}) and described large deviations: $\P(\tau'_u < \log u/\rho)$ for $\rho > \rho_0$ and
$\P(\tau'_u > \log u/\rho)$ for $\rho < \rho_0$.  Recently Buraczewski and Ma\'slanka \cite{BM} essentially improved his results applying the techniques introduced in the present paper. Pointwise estimates of $\tau'_u$ analogous to \eqref{eq: aymp} were obtained under hypotheses \eqref{H1} and \eqref{H2} only.

 The general shape of \eqref{H2.3} matches with the
 previous results \cite{BCDZ} but now $\tau_u$ is much better localized. Moreover, contrary to \cite{BCDZ}, Theorem \ref{thm: main theorem} is valid for all $\rho>0$. But there is a  price to pay:  considerably stronger assumptions than those of~\cite{BCDZ}:
\begin{itemize}
%\item In Theorem \ref{thm: main theorem} we identify, in particular, the most %likely passage time of $Y_n$ into the set $(u,\8)$. This happens when the %parameter $\ov\a$ is the smallest, that is when $\a=\a_0$ and  $\ov \a = \a_0$ %(see the figure above).
%\item The function $\Theta(u)$ is just a correction term, which is needed %because $\tau_u$ is a discrete time, whereas $k_u$ depends on $u$ in a %continuous way. However in the proof we will just define $k_u$ as above and, %without any saying, assume that $k_u$ is a natural number. Thus, the term %$\lambda(\a)^{\Theta(u)}$ will not appear in our proof.
\item Hypothesis \eqref{H1.4} indicates the optimal set of indices. Indeed, if $\a _{\min} \leq 1$ then \eqref{eq: aymp} holds for every $\rho>0$ and $\a > \a _{\min} $. If not, we require the condition \eqref{H2.5},
%\begin{equation}\label{lambda} \Lambda (1)<\Lambda (\a ),\end{equation}
which is well justified by Theorem \ref{mthm2} that provides a class of appropriate counterexamples to \eqref{eq: aymp}.

If $\a_{\min} > 1$ then there is $1<\tilde \a < \a _0$ such that $\Lambda (1)=\Lambda (\tilde \a )$,  \eqref{H2.5} is satisfied for $\a >\tilde \a $, and we have \eqref{eq: aymp} for $\rho = \Lambda'(\a) > \Lambda'(\tilde \a)$. If $\rho = \Lambda'(\a) < \Lambda'(\tilde \a)$ the asymptotic may be of different order as we can see in Theorem \ref{mthm2}.

%in hypothesis \eqref{H1.4} a further assumption on $\a$. On the first sight %this seems to be artificial and needed just by the techniques applied in our %proof. Nevertheless this is an optimal set of indices under which Theorem %\ref{thm: main theorem} holds and as we will see below in Theorem \ref{mthm2}, %if this condition is not satisfied one can find a class of counterexamples to %\eqref{eq: aymp}.
%\item The hitting time $\tau_u$ is expressed in terms of the derivative of $\Lambda$. Thus to find $\a$ %for which this relation holds and in particular $\Lambda'(\a)$ is positive we must have that %$\a>\a_{\min}$. In particular, then \eqref{H2} implies that $\a_{\min}$ is well defined.
%    \item Regardless of the value of $\a_{\min}$ the result hold for every $k_u %< \frac{\log u}{\Lambda '(\a_0)}$, of course assuming \eqref{H2} for an %appropriate $\a$.
  %  If $\a_{\min} > 1$  there is $1<\tilde \a < \a _0$ such that $\Lambda %(1)=\Lambda (\tilde \a )$ and \eqref{lambda} is satisfied for $\a >\tilde \a $ %and we have \eqref{eq: aymp} for $k_u<\frac{\log u}{\Lambda '(\tilde \a  )}$.   %Theorem \ref{thm: main theorem} does not describe the asymptotics for $k_u\geq %\frac{\log u}{\Lambda '(\tilde \a )}$ (i.e $\a {\min} <\a \leq \tilde \a $). In %the latter case, as we will see below in Theorem \ref{mthm2}, the asymptotic %may be of different order.
     \item Assumption \eqref{H1.3} is needed to ensure that the  processes $Y_n$  exceeds with positive probability the level $u$ for an arbitrary large $u$
     and this is the weakest assumption implying that (see \cite{BDM}, Proposition 2.5.4). Then, as explained in \cite{BD}, the constants $c_0$ in \eqref{eq: twostar} and $c(\a)$ in \eqref{eq: aymp} are strictly positive.

          \item Assumptions \eqref{H1.6} and \eqref{H1.2} are technical. The strategy of the proof requires that $\mu $ - the joint law of $(A,B)$ is dominated by a product {law and that the distribution} of $A$ has a density that decays properly at $+\infty $.
\item  The function $\Theta(u)$ in \eqref{eq: aymp} is  a correction term, which is needed because $\tau_u$ is  discrete, whereas $k_u$ depends on $u$ in a continuous way. Nevertheless, to simplify the notation, we  avoid writing the integer part and  below $k_u$ always denotes an integer number.

\item The assumption \eqref{H2.3} on $g_A$  can be improved to:   $g_A da$ is a measure containing non-trivial absolutely continuous part. We comment more on it
in Remark \ref{rem: 8.1}  at the end of Section \ref{sec: thm2}. %\ref{Proof of Theorem \ref{mthm2}}.

\end{itemize}

To prove the main results we analyze  path properties of perpetuity sequence $Y_n$. This method, although technically quite involved, is ultimately very rewarding and finally it provides a much deeper insight into the process. This strategy has been used by  Enriquez et al. \cite{ESZ} and Collamore, Vidyashankar \cite{CV} to obtain an explicit formula for the  limiting constant $c_+$ in \eqref{eq: twostarrrr} and also by Buraczewski et al. \cite{BDMZ} to prove large deviations results.

\subsection{Acknowledgements} We thank the referees for useful remarks which helped us to improve the previous version of the paper.

\section{Large deviations for random walks}

 To analyze the behavior of the random walk $\{ \Pi_n \}$,
the following uniform large deviation theorem, due to Petrov \cite{Petrov}, Theorem 2, will play a key role (see also the Bahadur-Rao theorem in \cite{DZ}).

\begin{lem}[Petrov]
\label{thm:3.1Petrov} Assume that the law of $\log A$ is nonlattice and
 that $c$ satisfies $ {\mathbb E} \left[ \log A \right] < c < s_0$, and suppose that $\delta(n)$ is an arbitrary function satisfying  $\lim_{n \to \infty} \delta(n) = 0$.
Then with %$\alpha \equiv \alpha(a^{-1})$ i.e.
$\alpha$ chosen such that $\Lambda '(\alpha )=c$, we have that
\begin{align*} %\label{petrov-1}
{\mathbb P} \big\{ \Pi_n > &\ e^{n(c + \gamma_n)} \big\}\\
&  = \frac{1}{\alpha\sigma(\alpha) \sqrt{2\pi n}} \exp\left\{-n \Big( \alpha(c+\gamma_n) - \Lambda(\alpha) + \frac{\gamma_n^2}{2\sigma^2(\alpha)}
  \left(1 + O(|\gamma_n| \right) \Big) \right\} (1+o(1))
\end{align*}
as $n \to \infty$,
uniformly with respect to $c$ and $\gamma_n$ in the range
\begin{equation} \label{petrov-2}
{\mathbb E} \left[ \log \,A \right] + \epsilon \le c \le s_0 - \epsilon \quad \text{\rm and} \quad |\gamma_n| \le \delta(n),
\end{equation}
where $\epsilon >0$.
\end{lem}

\begin{rem}{\rm
In \eqref{petrov-2}, we may have that $s_0= \infty$ or ${\mathbb E} \left[ \log \,A \right] =-\infty$.  In these cases, the quantities
$\infty -\epsilon$ or  $-\infty -\epsilon$ should be interpreted as arbitrary positive, respectively negative, constants.}
\end{rem}

In fact we will also use some refinements of the last result

\begin{lem}
\label{lem: petrov2.0}
Under the hypotheses of the previous theorem assume additionally that $\d_n,j_n$ satisfy
\begin{equation}
\label{eq: dn}
\max\big\{ \sqrt n|\d_n|,  j_n/\sqrt n\big\} \le \d (n).
\end{equation}
Then
%\begin{multline*}
$$\P\big[  \Pi_{n-j_n} \ge t e^{n\d_n} \big] = \frac 1{\a \s(\a)\sqrt{2\pi n}} \cdot t^{-\ov \a} e^{- \a n \d_n} e^{-j_n \Lambda(\a)}
(1+o_1(1))
\qquad \mbox{ as } n\to\8
$$
%\end{multline*}
uniform for all $\Lambda'(\a)$ satisfying  $\E \log A_1 + c_1 \le \Lambda'(\a) \le c_2$, $c_1, c_2>0$, $t=\exp (n\Lambda'(\a))$ and for all $\d_n,  j_n$ satisfying \eqref{eq: dn}.
\end{lem}
\begin{proof}
Let $\bar n=(n-j_n)$ and $\rho =\Lambda '(\a )$. We write $te^{n\d _n}=e^{n\rho + n\d _n}$ and $n\rho + n\d _n=(n-j_n)\rho +j_n\rho + n\d _n= \bar n\rho +\bar n (\frac{j_n\rho }{\bar n} + \frac{n\d _n}{\bar n})$. Now we may apply Lemma \ref{thm:3.1Petrov} with $\d '_n=\frac{j_n\rho }{\bar n} + \frac{n\d _n}{\bar n}$ playing the role of $\gamma _n$ and $\bar n$ playing the role of $n$. Clearly then
$$
\bar n \bigg (\frac{j_n\rho }{\bar n} + \frac{n\d _n}{\bar n}\bigg )^2=O( \d (n)).$$
Hence we may write
\begin{align*}
\P\big[  A_1..A_{n-j_n} \ge t e^{n\d_n} \big] &= \frac 1{\a \s(\a)\sqrt{2\pi n}} \cdot e^{-\rho \bar n\ov \a} e^{- \a \bar n \d '_n} (1+o_1(1))\\
&=\frac 1{\a \s(\a)\sqrt{2\pi n}} \cdot e^{-\rho n\ov \a} e^{- \a n \d _n}e^{-j_n\Lambda (\a )} (1+o_1(1))
\end{align*}
which proves the lemma.
\end{proof}

%%%%%%%%%%%%%%%%%%%%%%%%%%%%%%%%%%%%%%%%%%%%%%%%%%%%%%%%%%%%%%%%%%%%%%%%%%%%%%%%
\section{Lower estimates in Theorem \ref{thm: main theorem}}

In this section we prove lower estimates. Our main result is the following.
\begin{prop} Assume \eqref{H1}, \eqref{H2} and \eqref{H1.3}  %hypothesis of Theorem \ref{thm: main theorem}
are satisfied, then there is $\eta >0$ such that
\label{lem: lower estimates large n}
$$
 \P[\tau_u =  k_u  +1]  = \P\big[ M_{k_u}<u \mbox{ and \ } Y_{k_u+1}>u \big] \ge \frac{\eta}{\sqrt{\log u} }\ u^{-\ov \a}
$$ for $k_u = \frac{\log u}{\Lambda'(\a)}$ and sufficiently large $u$, where $M_n$ denotes the sequence of maxima of $Y_n$:
$$
M_n = \max\{0, Y_1, \ldots, Y_n\}.
$$
\end{prop}

The proof of this Proposition is based on two lemmas.
First we consider  the following perpetuity
$$
\ov Y_n = \sum_{j=1}^n \Pi_{j-1}|B_j|
$$ and we study the joint distribution of $\Pi_n$ and $\ov Y_n$ as $n \to \infty.$
 As we will see below we are able to control probability of the event
$$\big\{ c_1 u < \Pi_{k_u-j} < c_2 u,\ \ov Y_{k_u-j} < \gamma u  \big\},$$
for some constants $\gamma <1, c_1, c_2, j$. Next we can choose $((A_{k_u-j+1}, B_{k_u-j+1}), \ldots, (A_{k_u+1}, B_{k_u+1}))$ with positive probability (depending on $c_1,c_2,\gamma,j$, but not on $u$)
%special values of $((A_{k_u-j+1}, B_{k_u-j+1}), \ldots, (A_{k_u+1}, B_{k_u+1}))$
such that the corresponding perpetuity exceeds the level $u$ exactly at time $k_u+1$. Observe that this part of the proof does not require our continuity hypotheses \eqref{H1.6} and \eqref{H1.2}. 
\begin{lem}
\label{lem: 5.9}
Assume \eqref{H1} and \eqref{H2}. For a
fixed $r_0>1$, there is $s \geq 1$, such that for every $\gamma >0$, $1< r \leq r_0$ and $u\geq u(\gamma, r )$
\begin{equation} \label{aaa1}
\P\big\{ \g u< \Pi_{k_u} \le \g r u, \quad \! \ov Y_{k_u} \le \g s u \big\} \ge
\frac{D(r)\gamma^{-\a} }{\sqrt{\log u}} u^{-\ov \a},
\end{equation}
where $k_u = \frac{\log u}{\Lambda'(\a)}$ and the function $D(r)>0$ for $1<r\leq r_0$ is increasing.
\end{lem}
\begin{rem}
For given $r_0$ and $s$ the value of  $u(\gamma, r )$  is not uniformly bounded in $\gamma, r \to 0$.
The property that $s$ can be chosen independently of $\gamma>0$ and $1<r\leq r_0$ is crucial for our proof of Proposition \ref{lem: lower estimates large n}.
%In fact we will use this lemma for $\gamma $ from the fixed compact subset of $R_+$.
%The precise behaviour of $D(r)$ could be obtained from our proof, but is not important for .
%the above estimate holds for every $\g >0$, but it is not uniform in $u$. That is it holds for $u$ sufficiently large depending on $\g$.

\end{rem}
\begin{proof}
By Lemma \ref{thm:3.1Petrov}, there exists a constant $D(r)$ such that
\begin{equation}
\label{eq: wt1}
\P\big\{\gamma u < \Pi_{k_u} \le \g r u  \big\} =\frac{D(r)}{\sqrt{\log u}} \gamma^{-\a} u^{-{\ov \a}} (1+o(1)) \quad \mbox{as} \quad u \to \infty.
\end{equation} Recall $u^{-\ov \a} = u^{-\a} (\E A^\a)^{k_u}$.
Notice we have
\begin{multline*}
\P\big\{ \g u\le \Pi_{k_u} \le \g r u, \quad \! \ov Y_{k_u} \le \g s u \big\} \\
=
%\P\big\{ \g u\le \Pi_{k_u} \le \g r u\big\}
%- \P\big\{ \g u\le \Pi_{k_u} \le \g r u, \quad \! |Y_{k_u}| > \g s u \big\}\\ \ge
\P\big\{ \g u\le \Pi_{k_u} \le \g r u\big\}
- \P\big\{ \g u\le \Pi_{k_u} \le \g r u, \quad \! \overline Y_{k_u} > \g s u \big\}.\\
\end{multline*}
We are going to estimate the second summand in the
last expression and prove
\begin{equation}
\label{eq: d1}
\P\big\{ \g u\le \Pi_{k_u} \le \g r u, \quad \! \overline Y_{k_u} > \g s u \big\} \le
\frac{C  \g^{-\a}}{s^{\eps} }u^{-\ov \a}
\cdot \bigg(
\frac{ D(r) }{\sqrt{\log u} } +\frac 1{{ \log u} }
\bigg)
\end{equation}
for $s\geq 1$ and $u\geq u(\g, r, s)$.
Then we take $s=1+(2C)^\frac{1}{\eps}$,  increase $u(\g, r, s)$ if necessarily and conclude the Lemma.

\medskip

Notice that to prove \eqref{eq: d1}, without any loss of generality, we may assume that $|B_k|>1$ a.s.
We write
$$ \P\big\{ \g u\le \Pi_{k_u} \le \g r u,  \overline Y_{k_u} > \g s u \big\}
 \le \sum_{i\ge 0} \P\bigg\{ \g u\le \Pi_{k_u} \le \g r u,  \Pi_{k_u - i - 1}|B_{k_u -i }| > \frac{\g s u}{2 (i+1)^2} \bigg\}
 $$
We take large $K$ and we divide the sum into two parts depending whether $i>K\log k_u$ or $i\leq K\log k_u$.

\medskip
{\bf Case 1.} Suppose that $i>K\log k_u$. We take $\b<\a$ and define $\eps = \a - \b$, $\d = \lambda(\b)/\lambda(\a)<1$. Moreover, let $\Pi '_i=A_{k_u-i+1}\cdots A_{k_u}$. We write
\begin{align*}
% \P\big\{ \g u\le \Pi_{k_u} \le \g r u, \quad \! \overline Y_{k_u} > \g s u \big\}\\
% \le
  \P\bigg\{ \g u \le &\  \Pi_{k_u} \le \g r u,  \quad \! \Pi_{k_u - i - 1}|B_{k_u -i }| > \frac{\g s u}{2 i^2} \bigg\}\\
 &\le  \sum_{m\ge 0}
 \P\bigg\{  \Pi_{k_u} \ge \g  u, \quad \!  \frac{\g s u e^{m}}{2 i^2} \le \Pi_{k_u - i - 1}|B_{k_u -i }| < \frac{\g s u e^{m+1}}{2 i^2} \bigg\}\\
 &=\sum_{m\ge 0}
 \P\bigg\{  \Pi_{k_u-i-1}A_{k_u-i}\Pi '_i \ge \g  u, \quad \!  \frac{\g s u e^{m}}{2 i^2} \le \Pi_{k_u - i - 1}|B_{k_u -i }| < \frac{\g s u e^{m+1}}{2 i^2} \bigg\}\\
&\le  \sum_{m\ge 0}
\int \P\bigg\{  \Pi_{k_u-i-1} a \Pi '_i \ge \g  u, \quad \!  \frac{\g s u e^{m}}{2 i^2} \le \Pi_{k_u - i - 1}b < \frac{\g s u e^{m+1}}{2 i^2} \bigg\}\mu(da\, d|b|)\\
 &\le  \sum_{m\ge 0}
\int \P\bigg\{  \Pi _{k_u-i-1} \ge  \frac{\g s u e^m}{2i^2 b}\bigg\}
\cdot
\P\bigg\{  \Pi '_{i} \ge \frac{2i^2 b}{as e^{m+1}} \bigg\}
\mu(da\, d|b|)\\
 \end{align*}

 Now we apply twice the Chebyshev inequality and estimate the last expression by
\begin{align*}
  \sum_{m\ge 0}
\int\P\bigg\{  \Pi_{k_u-i-1} \ge & \ \frac{\g s u e^m}{2i^2 b}\bigg\} \cdot \P\bigg\{  \Pi_{i} \ge \frac{2i^2 b}{as e^{m+1}} \bigg\}\mu(da\, d|b|)\\
& \le  \sum_{m\ge 0}
\int  \frac{(2i^2 b)^\a}{(\g s u e^m)^\a}\, \l(\a)^{k_u-i}
\cdot \frac{(as e^m)^\b}{(2i^2 b)^\b} \l(\b)^i \mu(da\, d|b|)\\
&\le \frac {2^{\a -\b}}{\g^\a s^{\eps}} \frac 1{u^{\ov \a}}  i^{2\eps}\d^i\sum_{m\ge 0} e^{-m \eps}\int  b^{\eps} a^{\b}\mu(da\, d|\b|)\\
&= \frac{C  i^{2\eps}\d^i }{\g ^{\a}s^{\eps}} u^{-\ov \a}
\end{align*}
with the constant $C$ depending only on $\a , \b $ and $\mu $.
Summing over $i$ we obtain
\begin{align*}
\sum_{i>K \log k_u}   \P\bigg\{ \g u\le \Pi_{k_u} \le \g r u,& \quad \! \Pi_{k_u - i - 1}|B_{k_u -i }| > \frac{\g s u}{2 i^2} \bigg\}\\
&\le \sum_{i>K \log k_u} \frac{C  i^{2\eps}\d^i }{\g ^{\a}s^{\eps}} u^{-\ov \a}
\leq \frac{C}{s^{\eps}\g ^{\a}{\log u}} u^{-\ov \a}
%\frac{C_0  i^{2\eps}\d^i }{\sqrt{k_u}} u^{-\ov \a}
%\leq \frac{C  \eps(k_u) }{\sqrt{k_u}} u^{-\ov \a}
\end{align*}
provided $K$ is sufficiently large. Note that we can choose $K$ depending only on $\mu, \alpha $ and $\beta$.

{\bf Case 2.} Now we assume that $i\le K\log k_u$ and that $L$ is large enough and satisfy
\begin{equation}
\label{eq: L ineq}
\begin{split}
-\a L +1 < 0&\quad \mbox{if }\ \Lambda(\a)\ge 0,\\
-\a L +1 -\Lambda(\a) K < 0&\quad \mbox{if }\ \Lambda(\a)<  0.\\
\end{split}
\end{equation}
Then
\begin{equation}
\label{eq: 34}
\begin{split}
\P\bigg\{\g u \leq \Pi_{k_u} \leq &\  \g  ru, \quad \! \Pi_{k_u - i - 1}|B_{k_u -i }| > \frac{\g s u}{2 (i+1)^2} \bigg\}\\
& \le
\P\bigg\{  \Pi_{k_u - i - 1}|B_{k_u -i }| \ge  \g s u\cdot k_u^L \bigg\}\\
& \quad + \P\bigg\{   \g u \leq \Pi_{k_u} \leq \g  ru    ,  \,     \frac{\g s u}{2 (i+1)^2} \le  \Pi_{k_u - i - 1}|B_{k_u -i }| \le \g s u \cdot k_u^L \bigg\}\\
\end{split}
\end{equation}
The first term is asymptotically negligible, since by the Chebyshev inequality and \eqref{eq: L ineq} (then the term in the brackets below is uniformly bounded for $i\le K \log k_u$)
\begin{equation}
\label{eq: resasym}
\begin{split}
\P\big\{  \Pi_{k_u - i - 1}|B_{k_u -i }| \ge  \g s u\cdot k_u^L \big\}
&\le \P\big\{  \Pi_{k_u - i - 1} \ge  \g s u\cdot k_u^L \big\}\\
&\le \frac 1{\gamma^\a s^\a u^\a k_u^{\a L}} \lambda(\alpha)^{k_u-i-1}\E\big[|B|^{\a}\big] \\
&\le \frac { C\gamma^{-\a} s^{-\a}}{{\log u}} u^{-\ov \a}\Big( k_u^{- \a L + 1} e^{-(i+1) \Lambda(\a)} \Big)\\
&\le \frac { C}{\gamma^{\a} s^{\a}{\log u}} u^{-\ov \a}
\end{split}
\end{equation}
%because, by \eqref{eq: L ineq}, the term in the brackets is uniformly bounded for $i\le K \log k_u$.
We will use this estimate later on.

Next to estimate the second term in \eqref{eq: 34} we write
\begin{align*}
 \P\bigg\{    \g u \leq \Pi_{k_u} \leq &\ \g  ru,  \quad \!   \frac{\g s u}{2 i^2} \le  \Pi_{k_u - i - 1}|B_{k_u -i }| \le \g s u \cdot k_u^L \bigg\}\\
&=\P\bigg\{\g u \leq \Pi_{k_u-i-1}A_{k_u-i}\Pi '_i \leq \g  ru, \quad \!   \frac{\g s u}{2 i^2} \le  \Pi_{k_u - i - 1}|B_{k_u -i }| \le \g s u \cdot k_u^L \bigg\}\\
&\le \sum_{0\le m\le \log (2i^2 k_u^L)} \int \P \big( U(a,b,m) \big)\mu(da, d|b|)
\end{align*}
for
$$
U(a,b,m) = \bigg\{\g u \leq \Pi_{k_u-i-1}a\Pi '_i \leq \g  ru, \quad \!   \frac{\g s u e^m}{2 i^2} \le  \Pi_{k_u - i - 1}b \le \frac{ \g s u \cdot e^{m+1}}{2i^2} \bigg\}.
$$ Now we dominate the sets $U(a,b,m)$ as  follows
\begin{align*}
\P\big(U(a,b,m)\big) &\le
  \P\bigg\{\g u \leq \Pi_{k_u-i-1}a\Pi '_i \leq \g  ru, \quad \!   \frac{2bi^2}{s a e^{m+1}} \le  \Pi_{i}' \le    \frac{2r b i^2}{s a e^{m}}         \bigg\}\\
&\le  \int \P\bigg\{\frac{\g u}{a w} \leq \Pi_{k_u-i-1} \leq  \frac{\g r u }{a w}  \bigg\} \P\bigg\{ \frac{2bi^2}{s a e^{m+1}} \le  \Pi_{i}' \le    \frac{2r b i^2}{s a e^{m}} ,   \Pi_{i}' \in (w, w+dw)  \bigg\}\\
&\le  \sup_{\frac{2bi^2}{s a e^{m+1}} \le  w \le    \frac{2r b i^2}{s a e^{m}}}\P\bigg\{\frac{\g u}{a w} \leq \Pi_{k_u-i-1} \leq  \frac{\g r u }{a w}  \bigg\} \P\bigg\{ \frac{2bi^2}{s a e^{m+1}} \le  \Pi_{i}' \le    \frac{2r b i^2}{s a e^{m}} \bigg\}
%\le \sum_{0\le m\le \log (2i^2 k_u^L)} \int \P\bigg[ \Pi_{k_u - i - 1} \ge \frac{\g sue^m}{2 i^2 b} \bigg]
 %\P\bigg[ \Pi_{i} > \frac{bi^2}{as e^m} \bigg] \mu(da,\,d|b|)
\end{align*}
 Since both $m$ and $i$ are bounded by a constant times $\log k_u$, we can apply Petrov's result  \ref{thm:3.1Petrov} on the set
$$
\Theta = \big\{ b\le e^{\sqrt{k_u}}  \big\}
$$
{\bf Case 2a.} %Define constants: $\b<\a, \eps = \a - \b$, $\d = %\lambda(\b)/\lambda(\a)<1$. Recall $i< K \log k_u$.
Applying Lemma  \ref{thm:3.1Petrov} and Chebyshev inequality we have
\begin{align*}
\sum_{0\le m\le \log (2i^2 k_u^L)}& \int  {\bf 1}_{\Theta}(b) \P \big( U(a,b,m) \big)\mu(da, d|b|)\\
&\le \!\!\!\sum_{0\le m\le \log (2i^2 k_u^L)}\!\! \int \!\! {\bf 1}_{\Theta}(b)
\frac {C D(r)}{\sqrt{k_u-i-1}} \frac{(i^2 b)^\a}{(\g s e^m)^\a} \frac{\lambda(\a)^{k_u-i}}{u^\a}  \frac{(as e^m)^\b}{(b i^2)^\b}
\lambda(\b)^i
 \mu(da,\,d|b|)\\
&\le \frac{C D(r)}{\g^\a s^{\eps}} \frac 1{\sqrt {k_u}} u^{-\ov \a} i^{2\eps}\d^i \sum_{m\geq 0} e^{-m\eps} \cdot \int b^{\eps} a^\b \mu(da,\,d|b|)\\
&\le \frac{C D(r)}{\g^\a s^{\eps}\sqrt {k_u}} u^{-\ov \a}
%\frac {i^{2\eps}\d^i}{\sqrt {k_u}} u^{-\ov \a}.\\
\end{align*}
with the constant $C$ depending only on $\a , \b$ and $\mu $.

\noindent
{\bf Case 2b.} Applying twice the Chebyshev inequality we obtain
\begin{align*}
\sum_{0\le m\le \log (2i^2 k_u^L)}& \int  {\bf 1}_{\Theta^c}(b) \P \big( U(a,b,m) \big)\mu(da, d|b|)\\
&\le \sum_{0\le m\le \log (2i^2 k_u^L)} \int {\bf 1}_{\Theta^c}(b)
 \frac{(2i^2 b)^\a}{(\g s e^m)^\a}\cdot \frac{\lambda(\a)^{k_u-i}}{u^\a} \cdot \frac{(as e^m)^\b}{(b i^2)^\b}\cdot
\lambda(\b)^i
 \mu(da,\,d|b|)\\
&\le \frac{2^{\a}}{\g^\a s^{\eps}}  u^{-\ov \a} i^{2\eps}\d^i \sum_{m\geq 0} e^{-m\eps} \cdot \int {\bf 1}_{\Theta^c}(b) b^{\eps} a^\b \mu(da,\,d|b|)\\
&\le \frac{C}{\g^\a s^{\eps}}  u^{-\ov \a}\E |B|^{\eps }A^{\beta }{\bf 1}_{\Theta^c}
\end{align*}
We estimate the integral by the H\"older inequality with
$p_1=\frac{\alpha}{\b}, p_2=\frac{\alpha}{\epsilon}$, $\frac 1{p_1}+ \frac 1{p_2}=1$, applied to variables $A^\b$, $B^\epsilon {\bf 1}_{\Theta^c}$.
We get
$$
\int {\bf 1}_{\Theta^c}(b) b^{\eps} a^\b \mu(da,\,d|b|) \le
\E \big[ A^{\alpha}\big]^{\frac{\b}{\alpha}}   \E\big|B|^{\alpha}{\bf 1}_{\Theta^c} \big]^{\frac {\epsilon}{\alpha}}\leq
e^{-\frac{\d\epsilon}{\alpha}\sqrt{k_u}}\E \big[ A^{\alpha}\big]^{\frac{\b}{\alpha}}   \E\big|B|^{\alpha +\delta} \big]^{\frac {\epsilon}{\alpha}} \le \frac {C} {{k_u}}$$
The second inequality follows by Chebyshev inequality.

Finally, by \eqref{eq: resasym} and above estimates we obtain  the estimate in
case 2
\begin{align*}
\sum_{i \le K \log k_u} \P\bigg\{ \g u\le &\ \Pi_{k_u} \le \g r u,  \Pi_{k_u - i - 1}|B_{k_u -i }| > \frac{\g s u}{2 (i+1)^2} \bigg\}\\
&\le \frac{C}{\g^\a \sqrt{\log u}} \, u^{-\ov \a}
\sum_{i\le K \log k_u} \Big(\frac{ s^{-\a} k_u^{-\a L +1} e^{-i \Lambda(\a)}}{\sqrt{log u}} + 2D(r)s^{-\eps} i^{2\eps} \d^i )
\Big)\\
&\le \frac{C}{\g^\a {\log u}} \, u^{-\ov \a}s^{-\a } + \frac{CD(r)}{\g^\a \sqrt{\log u}} \, u^{-\ov \a}s^{-\epsilon}
\end{align*}
for $s\geq 1$. Combining both cases we obtain \eqref{eq: d1} and the lemma follows.
%By the choice of $L$ in  \eqref{eq: L ineq} we can take large  $s$ such that %the last sum is dominated
%by
%$$\frac 14\ \frac{C_0}{\g^\a \sqrt{\log u}} \, u^{-\ov \a}.$$

\end{proof}

\begin{lem}
\label{lem: mult}
Suppose that \eqref{H1.3} is satisfied. %Under hypotheses of Theorem \ref{thm: %main theorem},
Then there is $n$ such that
\begin{equation}
\label{eq: ws}
\P\big[ \Pi_n>1\ \mbox{ and }\ Y_n>0 \ \mbox{ and }\ Y_k < Y_n \ \mbox{ for } k=1,\ldots,n-1
\big] >0.
\end{equation}
In particular
$$
\P\big[ \Pi_n>1\ \mbox{ and }\ Y_n>0
\big] >0
 $$ for some $n$.
\end{lem}
\begin{proof} Here we use assumption \eqref{H1.3}.
 Of course the lemma holds for  $n=1$  when $b_2>0$, hence we assume in the proof  $b_2<0$. Then $b_1>0$.
We fix parameters $\d,N,M$ (their values will be specified below). Define
 $$
 U_\d(a,b) = \big\{ (a',b'):\; |a'-a| \le \d a  \ \mbox{ and }\ |b'-b|\le \d |b| \big\}
 $$
 and let
 \begin{align*}
 U = \Big\{ \big\{(A_k,B_k)\big\}_{k=1}^{N+M}:\
&(A_k, B_k) \in  U_\d (a_2,b_2)\ \mbox{ for }\ k=1,\ldots , N;\\
&(A_k, B_k) \in  U_\d (a_1,b_1)\ \mbox{ for }\ k=N+1,\ldots , N+M
 \Big\}
 \end{align*}
 for $\d< \min\{|b_2|, b_1\}$.
By assumption \eqref{H1.3} the  probability of $U$.  From now we  consider  the perpetuity
$$ Y_j = \sum_{i=1}^j A_1\cdots A_{i-1}B_i
= \sum_{i=1}^j \Pi_{i-1}B_i. $$ on the set $U$.
 We have
\begin{eqnarray*}
Y_{N+M} &=& Y_N + \Pi_N \sum_{j=1}^M a_1^{j-1} b_1\\
&\ge& \sum_{j=1}^N \big((1+\d)a_2\big)^{j-1} b_2 (1+\d) + \big( (1-\d) a_2 \big)^N \sum_{j=1}^M\big((1-\d)a_1\big)^{j-1} b_1 (1-\d) \\
&=& b_2(1+\d) \frac{(1+\d)^N a_2^N-1}{(1+\d)a_2 - 1} + (1-\d)^N a_2^N \frac{(1-\d)^M a_1^M - 1}{(1-\d)a_1 - 1} b_1 (1-\d).
\end{eqnarray*}
Denote the last expression by $f(\d)$.
We will find integers $N$ and $M$ such that $a_2^N a_1^M>1$ and $f(0)>0$. Then by continuity,
there exists $\d>0$ such that $f(\d)>0$ and simultaneously $\Pi_{N+M}=A_1...A_{M+N}>1$ for any $A_1,..., A_{M+N}$ in $U$.
We have
$$
f(0) = \frac{b_1}{1-a_1} (1-a_1^M) a_2^N + \frac {b_2}{1-a_2} (1-a_2^N).
$$
To prove that the last expression is strictly positive recall
$$
\frac{b_1}{ 1-a_1} > \frac{b_2}{1-a_2}>0.
$$ Since this is strict inequality and $a_1<1$ we can take large $M$ such that
$$
\frac{b_1}{ 1-a_1}(1-a_1^M) > \frac{b_2}{1-a_2}.
$$ Now, for any $N\geq 1$ we have
$$
\frac{b_1}{ 1-a_1}(1-a_1^M)a_2^N > \frac{b_2}{1-a_2}a_2^N > \frac{b_2}{1-a_2}(a_2^N-1)
$$ and this imply $f(0)>0$. We take $N$ large enough to satisfy  $a_1^M a_2^N > 1$.

\medskip

Notice that $Y_j<0$ for $j\le N$, hence $$Y_j<Y_{N+M}  \mbox { for } j=1,\ldots,N.$$
Moreover since for $j>N$
$$
Y_{j+1} - Y_j = \Pi_j B_{j+1} > 0
$$  the sequence increases for $j>N$ and attains its maximum for $j=N+M$. Therefore also
$$Y_j<Y_{N+M}  \mbox { for } j=N+1,\ldots,N+M-1.$$

Finally since
%\begin{align*}
$$
\P\big[ \Pi_{N+M}>1\ \mbox{ and }\ Y_{N+M}>0  \ \mbox{ and }\ Y_k < Y_{N+M} \mbox{ for } k=1,\ldots,{N+M}-1
\big] \\
%&\ge \P\big[
%(A_k, B_k)\in  U_\d (a_2,b_2) \mbox{ for } k=1,\ldots , N;\\
%&\quad (A_k, B_k) \in  U_\d (a_1,b_1) \mbox{ for } k=N+1,\ldots , N+M\big]
 \ge \P(U) > 0
$$
%\end{align*}
we conclude the lemma for $n=N+M$.
\end{proof}

\begin{proof}[Proof of Proposition \ref{lem: lower estimates large n}]  We will first prove the result under an additional assumption that
$\P[A>1, B>0]>0$.
Fix a point $(C_A,C_B)\in {\rm supp \mu}$ such that $C_A > 1$ and $C_B>0$.
Define
$$
\theta =\Big( 1 - \frac 34\eps_0 \Big)\bigg( \frac {C_B}{C_A^{j}}\ \frac{C_A^j-1}{C_A-1} \bigg)^{-1},
$$
Notice that for all $0<\eps _0\leq 1$ and $j\geq 2$ we have the inequalities   $\theta_1\geq \theta \geq \theta _0>0$, where  constants  $\theta_0,\theta_1$ depends on $C_A, C_B$ only.
Fix any   $\eps_0\le \frac{C_A-1}{4C_A}\le 1$ satisfying
$$
\theta_0 C_B > 2\eps_0.
$$
Let $s_0$ satisfy the conclusion of Lemma \ref{lem: 5.9}
for $r_0=4$.
We put $C_2 = s_0\theta_1$ and fix large enough $j$ such that
$$
\frac{C_2}{C_A^j} < \frac{\eps_0}{2}
$$
% $s_0=\frac{2C_2}{\theta_1}$

For very small $0< \d\leq \frac12$ (that will be defined slightly later) consider the set
\begin{multline*}
\Omega = \bigg\{ \frac{\theta u(1-\d)}{C_A^j} < \Pi_{k_u-j} <\frac{\theta u(1+\d)}{C_A^j},\ \ov Y_{k_u-j}  <\frac{C_2 u}{C_A^j}, \\
(A'_k, B'_k) \in U_\d( C_A , C_B) \ \mbox{ for }\; k=1,\ldots,j+1
\bigg\}
\end{multline*}
Notice that on the set $\Omega$ we have
\begin{eqnarray*}
\Pi_{k_u-j}Y'_{j}&=& \Pi_{k_u-j}\big( B'_1 + A'_1 B'_2 +\cdots + A'_1\ldots A'_{j-1} B'_j \big)\\
&\le& \frac{\theta u (1+\d)}{C_A^j} \cdot (1+\d)C_B\cdot \frac{C_A^j(1+\d)^j-1}{C_A(1+\d)-1}\\
&=& D_1(\d) u\\
\Pi_{k_u-j}Y'_{j}
&\ge& \frac{\theta u (1-\d)}{C_A^j}\cdot (1-\d)C_B \cdot  \frac{C_A^j(1-\d)^j-1}{C_A(1-\d)-1}\\
&=& D_2(\d) u\\
\Pi_{k_u}B_{k_u+1} &\ge&  \frac{\theta u (1-\d)}{C_A^j}\cdot C_A^j(1-\d)^j \cdot C_B(1-\d) = \theta C_B(1-\d)^{j+2}\\
&=& D_3(\d)u.
\end{eqnarray*}
moreover, by direct computation     %Notice that the functions $D_i(\d)$ are continuous and
\begin{eqnarray*}
D_1(0) &=& D_2(0) =  \frac {\theta C_B}{C_A^{j}}\cdot \frac{C_A^j-1}{C_A-1} = 1-\frac 34 \eps_0,\\
D_3(0) &=& \theta C_B \ge  \theta_0 C_B \ge 2\eps_0.
\end{eqnarray*}
Therefore, by continuity of   $D_i(\d)$     one can choose $\d \leq \frac 12$ such that
$$
D_1(\d) < 1-\frac{\eps_0}2,\qquad D_2(\d)>1-\eps_0, \qquad D_3(\d)\ge \frac 32 \eps_0.
$$
Hence we have
\begin{align*}
\Omega &\subset \Big\{ \Pi_{k_u-j}Y'_{j-1}\in\!\big( (1-\eps_0)u, (1-\eps_0/2)u \big), \ov Y_{k_u-j} <\frac{\eps_0}2u, \Pi_{k_u}B_{k_u+1}\ge  (3/2) \eps_0 u, M_{k_u}<u
\Big\}\\
&\subset \Big\{ Y_{k_u}\in \big( (1-3\eps_0/2)u,u \big),\ \Pi_{k_u}B_{k_u+1}\ge 3\eps_0 u/2,\  M_{k_u}<u \Big\}\\
&\subset \Big\{ M_{k_u}<u \mbox{ and }\  Y_{k_u+1} >u \Big\} ,
\end{align*}
%where $M_{k_u}$ is as in ...
On the other hand by Lemma \ref{lem: 5.9} applied with $$\gamma= \frac{\theta (1-\d)}{C_A^j},\quad r= \frac{1+\d}{1-\d}, \quad r_0=4, \qquad \mbox{and}\quad
s=\frac{C_2}{\theta(1-\d)}\geq s_0 =\frac{C_2}{\theta_1}$$ we obtain
\begin{multline*}
\P(\Omega) = \P\bigg\{ \frac{\theta u(1-\d)}{C_A^j} < \Pi_{k_u-j} <\frac{\theta u(1+\d)}{C_A^j},\ \ov Y_{k_u-j} <\frac{C_2 u}{C_A^j}
 \bigg\} \\
\cdot
\P\big\{ (A'_k, B'_k) \in U_\d( C_A , C_B) \ \mbox{ for }\; k=1,..,j+1
 \big\}
\ge \frac{\eta }{\sqrt{\log u}} \ u^{-\ov \a}
\end{multline*}
for some very small constant $\eta $.

\medskip

If $\P[A>1, B>0]=0$ we apply Lemma \ref{lem: mult} and proceed as above. This time
we fix a point $(C_A, C_B)\in {\rm supp } \mu^{*n}$ such that $C_A>1$ and $C_B>0$, but instead of choosing $(A_k', B_k')$ with the law $\mu $ close to $(C_A,C_B)$ one has to pick up $(\tilde A_k, \tilde B_k)$ with the law $\mu^{*n}$ (i.e. partial products and perpetuities). This means: %$(\Pi_{n,k}, Y_{n,k})$.
\begin{align*}
\tilde A_1=A' _1\cdots A '_n, &\quad \tilde B_1=\sum _{i=1}^n A' _1\cdots A '_{i-1}B' _i\\
\tilde A_k=A' _{(k-1)n+1}\cdots A '_{kn}, &\quad \tilde B_k=\sum _{i=(k-1)n+1}^{kn} A' _{(k-1)n+1}\cdots A '_{i-1}B' _i
\end{align*}
and $(\tilde A_k, \tilde B_k)$ are chosen accordingly.
Exactly the same calculations as above give the result.
The condition $\{Y_k<Y_n, k<n\}$ in \eqref{eq: ws} is needed, to ensure that the perpetuity will not exceed the level $u$ before time $k_u$.
We omit the details.
\end{proof}

%%%%%%%%%%%%%%%%%%%%%%%%%%%%%%%%%%%%%%%%%%%%%%%%%%%%%%%%%%%%%%%%%%%%%%%%%%%%%%%%
\section{Upper estimates in Theorem \ref{thm: main theorem}}

In this section we prove the following result, which gives upper estimates in Theorem \ref{thm: main theorem}
\begin{prop}
\label{prop: 6.1}
Assume that \eqref{H1},  \eqref{H2}, \eqref{H1.4},   \eqref{H1.6} and \eqref{H1.2} are satisfied.  Then there exists $C$ such that for every $u\geq 2$  we have

%Under assumptions of Theorem \ref{thm: main theorem}, for large $u$
$$
\P\big[ M_{k_u}<u,\; Y_{k_u+1}>u \big] \le\P\big[ Y_{k_u}<u,\; Y_{k_u+1}>u \big] \le \frac C{\sqrt{\log u}}\ u^{-\ov \a}
$$
Moreover  there is $\s<1 $ and $C >0$ %, which  do not depend on $u\geq C$ and
such that  for  every $\eps>0$ and $u\geq 2$ we have
\begin{equation}\label{62}
\P\big[ \tau_u = k_u +1\big]
%Y_{k_u}\in[(1-\eps)u,u],\; Y_{k_u+1}>u \big]
\le \frac {C(\eps^{1-\s}+u^{-\xi '})}{\sqrt{\log u}\ } \ u^{-\ov \a}.
\end{equation}
\end{prop}
 As in the previous Section we consider here the joint law of $(Y_{k_u},\Pi_{k_u}B_{k_u+1} )$ and our main effort is to estimate
$$
\P\big[  Y_{k_u}\in\big((1-\eps)u,(1-\eps/2)u\big), \Pi_{k_u}B_{k_u+1} > \eps u \big].
$$
First we need two  technical lemmas.

\begin{lem}
\label{lem: 33}
Assume that \eqref{H1} and \eqref{H2} are satisfied.  For any fixed (small)  $\delta >0$ there exist $C_\d$ such that for every $n\in \Z$, $\epsilon>0$ and $u\geq 2$  we have
\begin{equation}\label{eq:33}
\P\bigg[ \Pi_{k_u-1} B_{k_u} > \frac{\eps u}{e^n} \bigg] \le \frac {C_\d}{\eps^{\a}} \frac{e^{\a n+\d |n|}}{\sqrt{\log u}} u^{-\ov \a}
\end{equation}
\end{lem}
\begin{rem}
The above formula is meaningful only if the right hand side is smaller than 1 but it is useful to write the estimate in the unified way.
\end{rem}
\begin{proof}
We consider three cases making distinction depending on whether $n, \log |B_{k_u}|$ are bigger or smaller than $\sqrt{k_u}$.

\noindent
{\bf Case 1.}
Let first $\delta n> \sqrt{k_u}$. Then by the Chebychev inequality with exponent $\a $ we have
\begin{align*}
\P\bigg[ \Pi_{k_u-1} B_{k_u} > \frac{\eps u}{e^n} \bigg] &\le Ce^{k_u\L (\a)}\E |B_{k_u}|^{\a }\eps^{-\a} u^{-\a} e^{\a n}\\
&\le C\eps^{-\a} e^{\a n+\d |n|}e^{-\sqrt{k_u}} u^{-\ov \a}
\end{align*}
and \eqref{eq:33} follows.

\noindent
{\bf Case 2.}
If $\log |B_{k_u}|>\sqrt{k_u}$, we write
\begin{align*}
\P\bigg[ \Pi_{k_u-1} B_{k_u}>\frac{\eps u}{e^n}\bigg ]&\leq \sum _{m\geq  \sqrt{k_u}}
\P\bigg[ \Pi_{k_u-1}>\frac{\eps u}{e^{n+m+1}} \bigg ]\P\bigg[ B_{k_u}>e^m\bigg ]\\
&\leq Ce^{k_u\L (\a)}\eps^{-\a}u^{-\a} e^{\a n}\E |B_{k_u}|^{\a }\sum _{m\geq  \sqrt{k_u}}e^{-\d m}\\
&\leq C\eps^{-\a} e^{\a n}e^{-\d \sqrt{k_u}} u^{-\ov \a}
\end{align*}
which gives \eqref{eq:33}.

\noindent
{\bf Case 3.}
If $\log |B_{k_u}|\le \sqrt{k_u}$ and $n\le \sqrt{k_u}$, we use Petrov theorem and we obtain
\begin{equation}
\P\bigg[ \Pi_{k_u-1} B_{k_u} > \frac{\eps u}{e^n} \bigg] \le C \frac{u^{-\ov \a}}{\sqrt{\log u}}\eps^{-\a} e^{\a n}\E |B_{k_u}|^{\a },
\end{equation}
 which completes the proof.
\end{proof}
Let $B^+_n = \max\{B_n,0\}$ and $Y^+_n =\sum_{k=1}^n\Pi_{k-1}B_k^+$. Then of course $Y_n^+ \ge Y_n$.

\begin{lem}
\label{lem: 333}
Assume that \eqref{H1}, \eqref{H2} and \eqref{H1.4} are satisfied.  For any fixed (small)  $\delta >0$
there exist $C_\d$ and $\theta <1$
 such that for every $n\in \Z$, $\epsilon>0$, $c\geq 1$ and $u\geq 2$  we have
$$
\P\bigg[ \Pi_{k_u-1} B_{k_u} > \frac{\eps u}{e^n} \ \mbox{ and }  Y^+_{k_u-1}
 > \frac {c u}{e^n} \bigg] \le \frac {C_\d}{\eps^{\theta}} \frac{e^{\a n+\d |n|}}{\sqrt{\log u}} u^{-\ov \a}
$$
\end{lem}
\begin{rem}
The above formula is meaningful only if the right hand side is smaller than 1 but it is usuful to write the estimate in the unified way. The Lemma will be used with fixed $c$. The condition $c\geq 1$ can be, of course, replaced by $c\geq c_0>0$.
\end{rem}
\begin{proof}
Denote
$$
g(n,m) = \P\bigg[ \Pi_{k_u-1} B_{k_u} > \frac{\eps u}{e^n} \ \mbox{ and } \frac {c u e^m}{e^{n}} < Y^+_{k_u-1}
\le \frac {c u e^{m+1}}{e^n} \bigg]
$$
It is sufficient to prove that for $m\ge 0$, $n\in \Z $ and $u\geq const$
\begin{equation}
\label{eq: gnm}
g(n,m) \le \frac {C e^{-\d m}}{\eps^{\theta}} \frac{e^{\a n+\d |n|}}{\sqrt{\log u}} u^{-\ov \a}
\end{equation}
Indeed, then
\begin{eqnarray*}
\P\bigg[ \Pi_{k_u-1} B_{k_u} > \frac{\eps u}{e^n} \ \mbox{ and }  Y^+_{k_u-1}
 > \frac {c u}{e^n} \bigg]
& = & \sum_{m\ge 0} g(n,m) \\
&\le & \sum_{m\ge 0}
\frac {C e^{-\d m}}{\eps^{\theta}} \frac{e^{\a n+\d |n|}}{\sqrt{\log u}} u^{-\ov \a}\\  &\le &
\frac {C}{\eps^{\theta}} \frac{e^{\a n+\d |n|}}{\sqrt{\log u}} u^{-\ov \a}\\  \end{eqnarray*}
To estimate $g(n,m)$, for $j\ge 1$ we  define the set of indices
$$
W_j^u = \bigg\{1\leq i< k_u : \frac{cue^m}{e^n e^j} \le \Pi_{i-1} B^+_{i}\le \frac {cue^m}{e^n e^{j-1}}
\bigg\}
$$ On the set  $\big\{ \frac {c u e^m}{e^{n}}\le Y^+_{k_u-1} \big\}$
there is  some $j>0$, such that  the number of elements in the set $W_j^u$ must be greater then $\frac{e^j}{ 10 j^2}$.
Indeed, assume that such a $j$ does not exists, i.e. for every $j>0$, $\# W_j^u \leq \frac{e^j}{10j^2}$, then
\begin{align*}
Y ^+_{k_u-1}& = \sum_{i< k_u} \Pi_{i-1} B^+_{i} = \sum_{j>0}\sum_{i\in W_j} \Pi_{i-1} B_{i}^+\\
 &\le \sum_j \frac{e^j}{10j^2}\cdot \frac{cue^m}{e^n e^{j-1}}
< \frac{cue^m}{e^n} \sum_j \frac{e}{10j^2} \\ &< \frac{cue^m}{e^n}.
\end{align*}

Let
$$
K^u =\bigg\{ (j,m_1,m_2): j\ge 1, 1\le m_1 < k_u, m_1 + \frac{e^j}{10 j^2} < m_2 < k_u
\bigg\}.
$$
Then
\begin{equation}
\label{eq: dgnm}
g(n,m)
\le \sum_{(j,m_1,m_2)\in K^u} \P\big[U(j,m_1,m_2)\big],
\end{equation} for
$$
U(j,m_1,m_2) = \bigg\{
\frac {cue^m}{e^n e^j} \le \Pi_{m_i-1}B^+_{m_i}\le \frac {cu e^m}{e^n e^{j-1}}, \ i=1,2,\
\mbox{ and } \ \Pi_{k_u-1}B_{k_u} > \frac {\eps u}{e^n}
\bigg\}.
$$
Below we use the convention $\Pi_k$, $\Pi'_n$, $\Pi''_m$ to denote independent products of $A_j$'s of length $k$, $n$, $m$, respectively.
Then for any triple $(j,m_1,m_2)\in K^u$ we have
\begin{align*}
 \P\big[U(j,& m_1,m_2)\big]\\
&\le \int
 \P\bigg[
\frac {cu e^m}{e^n e^j} \le \Pi_{m_1-1}b_1\le \frac {cu e^m}{e^n e^{j-1}}, \
\frac {cu e^m}{e^n e^j} \le \Pi_{m_1-1}a_1 \Pi'_{m_2-m_1} b_2 \le \frac {cu e^m}{e^n e^{j-1}},\ \\
&\quad \Pi_{m_1-1} a_1 \Pi'_{m_2-m_1} a_2 \Pi''_{k_u - m_2} b > \frac{\eps u}{e^n}\bigg]{\bf 1}_{\{b_1 > 0, b_2>0, b>0\}}
\mu(da_1,db_1)\mu(da_2, db_2)\mu(da,db)\\
&\le  \int
 \P\bigg[ \Pi_{m_1-1}\ge \frac {cue^m}{b_1 e^n e^{j}}\bigg] \cdot
 \P\bigg[ \Pi_{m_2-m_1} >  \frac {b_1}{b_2 a_1 e}\bigg] \cdot
 \P\bigg[ \Pi_{k_u-m_2} >  \frac {\eps e^{j-1}b_2}{cb a_2 e^m}\bigg]\\&\quad \cdot {\bf 1}_{\{b_1 > 0, b_2>0, b>0\}}
\mu(da_1,db_1)\mu(da_2, db_2)\mu(da,db)
\end{align*}
Fix parameters: $\b_1 = \a-\eps_1$, $\b_2 = \a-\eps_2$ such that $\b _1, \b_2<1$,  $\rho_1 = \frac{\l(\b_1)}{\l(\a)}<1$, $\rho_2 = \frac{\l(\b_2)}{\l(\a)}<1$ and $\rho _1 >\rho _2 $. Here we use \eqref{H1.4}. If $\a _{min}< 1$, $\L '(\a )>0$
then we take $\a _{min}<\b _2<\b _1<\min \{ 1,\a\} $. If $\a _{min}\geq 1$ and $\l (1)<\l (\a )$ then there is $\tilde \a <1$ such that $\l (\tilde \a )=\l (\a )$ and so we can take $\b_1 = \tilde \a+\eps_1$, $\b_2 =\tilde \a+\eps_2$.
We apply the Chebyshev inequality with parameters $\a, \b_1,\b_2$ and so
\begin{equation}
\label{eq: star1}
\begin{split}
 \P\big[U(j,& m_1,m_2)\big]\\
 &\le
  \int
\frac{(b_1 e^n e^j)^\a}{c^\a u^\a e^{\a m}}\; \l(\a)^{m_1}\cdot \frac{(b_2 a_1 e)^{\b_1}}{b_1^{\b_1}}\; \l(\b_1)^{m_2-m_1} \frac{(c ba_2 e^m)^{\b_2}}{(\eps e^{j-1}b_2)^{\b_2}}\; \l(\b_2)^{k_u-m_2} \\
&\quad \cdot
 {\bf 1}_{\{b_1 > 0, b_2>0, b>0\}} \mu(da_1,db_1)\mu(da_2, db_2)\mu(da,db)\\
&\le   C\eps^{-\b_2} e^{\a n} e^{-\eps_2 m} u^{-\ov \a}
   e^{j\eps_2} \rho_1^{m_2-m_1}\rho_2^{k_u-m_2} \\ &\quad  \cdot   \int
b_1^{\eps_1} b_2^{\b_1-\b_2} a_1^{\b_1} a_2^{\b_2} b^{\b_2} {\bf 1}_{\{b_1 > 0, b_2>0, b>0\}}
\mu(da_1,db_1)\mu(da_2, db_2)\mu(da,db)\\
&\le   C\eps^{-\b_2} e^{\a n} e^{-\eps_2 m} u^{-\ov \a}
   e^{j\eps_2} \rho_1^{m_2-m_1}\rho_2^{k_u-m_2}     \E [|B_1|^{\eps _1}A_1^{\b_1}]%{\bf 1}_{\{B_1 > 0\}})
 \E [|B_2|^{\b_1-\b_2} A_2^{\b_2}]%{\bf 1}_{\{ B_2>0\}})
\E |B|^{\b_2}% {\bf 1}_{\{ B>0\}}
\end{split}
\end{equation}
The product of expectations is finite, because of the H\"older inequality and
\eqref{H2}. Hence it is sufficient to estimate
$$ \eps^{-\b_2} e^{\a n} e^{-\eps_2 m} u^{-\ov \a}
\sum_{(j,m_1,m_2)\in K^u}   e^{j\eps_2} \rho_1^{m_2-m_1}\rho_2^{k_u-m_2}.$$
Notice that the sum (of geometric sequence ) above is always dominated by its maximal term, that is by  $Ck_u^{2\eps_2}<Ck_u$. Assume first that $n>\sqrt{k_u}$. Then combining \eqref{eq: dgnm} with the estimates above
\begin{equation}
\begin{split}
g(n,m)
&\le \sum_{(j,m_1,m_2)\in K^u} \P\big[U(j,m_1,m_2)\big]\\
&\le \eps^{-\b_2} e^{\a n} e^{-\eps_2 m} u^{-\ov \a}k^2_u \\ &\leq
\eps^{-\b_2} e^{(\a +\d)n} e^{-\d \sqrt{k_u}}e^{-\eps_2 m} u^{-\ov \a}k^2_u \\
&\leq \eps^{-\b_2} e^{(\a +\d)n} e^{-\eps_2 m} o\bigg(\frac{u^{-\ov \a}}{\sqrt{\log u}}\bigg)
\end{split}
\end{equation}
 For the rest we fix $C_1$, we assume that $n\leq \sqrt{k_u}$ and we consider 4 cases
\begin{eqnarray*}
K_1^u &=& \Big\{  (j,m_1,m_2)\in K^u:\; e^{j/2}>C_1 \log k_u   \Big\},\\
K_2^u &=& \Big\{  (j,m_1,m_2)\in K^u:\; e^{j/2} \le C_1 \log k_u,\ m_2 < k_u - k_u^{\frac 14}   \Big\},\\
K_3^u &=& \Big\{  (j,m_1,m_2)\in K^u:\; e^{j/2}\le C_1 \log k_u,\ m_1 < k_u - 2k_u^{\frac 14},\ m_2 \ge k_u - k_u^{\frac 14}   \Big\},\\
K_4^u &=& \Big\{  (j,m_1,m_2)\in K^u:\; e^{j/2} \le C_1 \log k_u,\ m_1 \ge k_u - 2k_u^{\frac 14}  \Big\}.
\end{eqnarray*}
\noindent
{\bf Case 1.} In this case there is $C_2$ such that $m_2-m_1>\frac{e^j}{ 10 j^2}\ge 2C_2 \log k_u$. Hence
\begin{align*}
g(n,m) & \leq C\eps^{-\b_2} e^{\a n} e^{-\eps_2 m} u^{-\ov \a}
\sum_{(j,m_1,m_2)\in K_1^u}   e^{j\eps_2} \rho_1^{m_2-m_1}\rho_2^{k_u-m_2}
\\ &\le
C\eps^{-\b_2} e^{\a n} e^{-\eps_2 m} u^{-\ov \a} k_u \rho_1^{C_2 \log k_u}\\
&=  \eps^{-\b_2} e^{\a n} e^{-\eps_2 m}  o\bigg(\frac{u^{-\ov \a}}{\sqrt{\log u}}\bigg)
\end{align*}

\noindent
{\bf Case 2.} For the sum over $K_2^u$ we write
\begin{align*}
g(n,m) &\leq C\eps^{-\b_2} e^{\a n} e^{-\eps_2 m} u^{-\ov \a}
\sum_{(j,m_1,m_2)\in K_2^u}   e^{j\eps_2} \rho_1^{m_2-m_1}\rho_2^{k_u-m_2}
\\ & \le
C\eps^{-\b_2} e^{\a n} e^{-\eps_2 m} u^{-\ov \a} k_u (\log k_u)^{2\eps_2 } \rho_2^{k_u^{\frac 14}}\\
&=  \eps^{-\b_2} e^{\a n} e^{-\eps_2 m}  o\bigg(\frac{u^{-\ov \a}}{\sqrt{\log u}}\bigg)
\end{align*}

\noindent
{\bf Case 3.} Here $m_2-m_1>k_u^{\frac 14}$ and  reasoning as above
\begin{align*}
g(n,m) &\leq C\eps^{-\b_2} e^{\a n} e^{-\eps_2 m} u^{-\ov \a}
\sum_{(j,m_1,m_2)\in K_3^u}   e^{j\eps_2} \rho_1^{m_2-m_1}\rho_2^{k_u-m_2}
\\ &\le
C\eps^{-\b_2} e^{\a n} e^{-\eps_2 m} u^{-\ov \a} k_u^2  (\log k_u)^{2\eps_2 } \rho_1^{k_u^{\frac 14}}\\
&=  \eps^{-\b_2} e^{\a n} e^{-\eps_2 m}  o\bigg(\frac{u^{-\ov \a}}{\sqrt{\log u}}\bigg)
\end{align*}

\noindent
{\bf Case 4a.} On the set $\Theta = \{ b_1 \le e^{\sqrt{k_u}}\}$ and $ |n-m| \le \sqrt {k_u}$
 we estimate in a slightly different way the first term, that is
$$ \P\bigg[ \Pi_{m_1} \ge \frac{cue^m}{b_1 e^n e^j}  \bigg].
$$
We  use Lemma \ref{thm:3.1Petrov}
which gives
$$ \P\bigg[ \Pi_{m_1} \ge \frac{cue^m}{b_1 e^n e^j}  \bigg] \le \frac{C (b_1 e^n e^j)^{\a}}{c^\a u^\a e^{m\a}}\cdot \frac{\lambda(\a)^{m_1}}{\sqrt{m_1}}.
$$
Then, if $ |n-m| \le 2\sqrt{k_u} $
\begin{multline*}
\P\bigg[ \Pi_{k_u-1} B_{k_u} > \frac{\eps u}{e^n} \ \mbox{ and } \frac {cue^m}{e^{n}} < Y^+_{k_u-1}
< \frac {cu e^{m+1}}{e^n},\ B\in \Theta\   \bigg]\\
\le C \eps^{-\b_2} e^{\a n} \frac{u^{-\ov \a}}{\sqrt{\log u}} e^{-\eps_2 m} \sum_{(j,m_1,m_2)\in K_4^u} e^{j\eps} \rho_1^{m_2-m_1}\rho_2^{k_u-m_2}.
\end{multline*}
Since, by the definition of $K^u$ the last sum is bounded, we obtain the required estimates.

\medskip

\noindent
{\bf Case 4b.} If $m> k_u^{\frac 14}$ we use \eqref{eq: star1} and obtain
$$
g(n,m) \le C \eps^{-\b_2} e^{\a n} e^{-\eps_2m/2} o\bigg(\frac{u^{-\ov \a}}{\sqrt{\log u}}\bigg).
$$

\medskip

\noindent
{\bf Case 4c.} If $n\le - \sqrt{k_u}$  and $m < k_u^{\frac 14}$, then for large $u$,  $e^{m}\le e^{\frac{\d |n|}{2}}$ and hence
$$ g(n,m)  \le \P\bigg[ \Pi_{k_u-1} B_{k_u}>\frac{\eps u}{ e^n} \bigg]  \le \frac{ e^{-\eps_1 m} e^{\a n+\d|n|}}{\eps^{\theta}}\frac{u^{-\ov \a}}{\sqrt{\log u}}.
$$

\noindent
{\bf Case 4d. } On the set $b_1\in \Theta^c$ we can sharpen  \eqref{eq: star1}. We get an extra decay for corresponding expectation $\int {\bf 1}_{\Theta^c}(b_1) b_1^{\eps _1} a_1 ^{\b _1} \mu(da_1,\,db_1)$.

We use the H\"older inequality with
 $\frac 1{p_1}+ \frac 1{p_2}=1$, $p_1$ close to 1. Then
$$
\int {\bf 1}_{\Theta^c}(b) b_1^{\eps _1} a_1^{\b _1}\mu(da_1 ,\,db_1) \le
\E \big[ A^{p_1\b _1}|B|^{p_1\eps _1 }]^{\frac{1}{p_1}}\cdot \P [|B|>
e^{\sqrt{k_u}}]^{\frac {1}{p_2}}$$
and the first term is finite agian by the H\"older inequality. Moreover, by \eqref{H2}.
$$
\P [|B|>
e^{\sqrt{k_u}}]^{\frac {1}{p_2}}\leq (\E |B|^{\a })
e^{-\sqrt{k_u}\a \slash p_2}$$
 and we deduce as above.

\medskip

Combining all the cases we obtain \eqref{eq: gnm} and complete the proof of the Lemma.
\end{proof}

\begin{proof}[Proof of Proposition \ref{prop: 6.1}]
{
We are going to show that for every $0<\eps < 1$
\begin{equation}
\label{eq: 6.1}
\P\big[  Y_{k_u}\in\big((1-\eps)u,(1-\eps/2)u\big), Y_{k_u+1}> u \big] \le \eps^{1-\sigma} \frac{C u^{-\ov \a}}{\sqrt{\log u}}.
\end{equation}
 \eqref{eq: 6.1} implies \eqref{62}.
Moreover, applying \eqref{eq: 6.1} to $\eps =2^{-n}$, $n=0,1,2...$ and summing up over $n$ we obtain Proposition \ref{prop: 6.1}.
Let $J_{\eps} = (1-\eps,1-\eps/2)u$.  %$J'_{\eps}= (1-2\eps,1)u$ large $u$.
We write $Y_{k_u} = B_1 + A_1 Y'_{k_u-1}$ where
$Y'_{k_u-1} = B_2 + A_2 B_3 + \cdots + A_2\ldots A_{k_u-1}B_{k_u}$. We will also use notation $\Pi '_{k_u-1}=A_2\ldots A_{k_u}$. We have % Then for large $u$
$$\P\big[  Y_{k_u}\in J_{\eps}, \Pi_{k_u} B_{k_u+1}>\eps u \big]
=
%\le
%\P\big[ |B_1|>u^\d\ \mbox{ and }\  \Pi_{k_u} B_{k_u+1}>\eps u \big]+
\P\big[ B_1+ A_1  Y'_{k_u-1}\in  J_{\eps}, \Pi_{k_u} B_{k_u+1}>\eps u \big]$$
%To estimate the first expression we apply the Chebyshev inequality and the H\"older inequality (with parameters $p=\frac{\a+\d}{\a}$ and  $p=\frac{\a+\d}{\d}$)
%\begin{eqnarray*}
%\P\big[ |B_1|>\d\ \mbox{ and }\  \Pi_{k_u} B_{k_u+1}>\eps u \big]
%&\le& \int {\bf 1}_{\{|b| > u^\d\}}\P[a A_2\ldots A_{k_u}B_{k_u+1} > \eps u]\mu(da,db)\\
%&\le& \int {\bf 1}_{\{|b| > u^\d\}}a^\a u^{-\ov \a}\mu(da,db)\\ &\le& C u^{-\ov \a}\big(\E A^{\a+\d}\big)^{\frac 1p} \P[|B|>u^\d]^{\frac 1q}\\
%\\
%&\le&  C u^{-\ov \a}\cdot u^{-\d_1} = o\bigg( \frac{u^{-\ov \a}}{\sqrt{\log u}}
%\bigg)
%\end{eqnarray*}
%The last inequality follow from the  H\"older inequality and \eqref{H1.1}

%Now, we estimate the second expression, by Lemma \ref{lem: 333}
and
\begin{equation}
\label{eq: diffomproof}
\begin{split}
\P\big[ B_1+ & A_1 Y'_{k_u-1} \in  J_{\eps}, \Pi_{k_u} B_{k_u+1}>\eps u \big]\\
&\le \sum_{n\in\Z} \P\Big[ e^n\le A_1 < e^{n+1}, B_1+\ A_1 Y'_{k_u-1}\in  J_{\eps}\ \mbox{ and }\  \Pi '_{k_u-1}B_{k_u+1} > \frac {\eps u}{e^{n+1}}\Big]\\
&\le \sum_{n\in\Z} \int {\bf 1}_{\{ \frac{u}{3e^{n+1}} < s < \frac {3u}{e^n}\}}\ \P\Big[e^n\le A_1 < e^{n+1},  A_1\in \frac 1s ( J_{\eps}-B_1) \Big]\\
&\qquad \cdot \P\Big[ \Pi '_{k_u-1}B_{k_u} > \frac{\eps u}{e^{n+1}}, Y'_{k_u-1}\in ds\Big]
\end{split}
\end{equation}
Notice that for $s\in(\frac{u}{3e^{n+1}},\frac {3u}{e^n})$ the interval $\frac 1{s} (J_{\eps}-B_1)$ has length at most $\frac 32\eps e^n$.
% and if $A_1\in \frac 1s  (J_{\eps}-B_1)$ then $C_1 e^n < A_1 < C_2 e^n$.
Thus, by Lemma \ref{lem: 333}
\begin{align*}
\P\big[ B_1+ & A_1 Y'_{k_u-1}\in  J_{\eps}, \Pi_{k_u} B_{k_u+1}>\eps u \big]\\
& \le \sum_{n\in \Z}
\int {\bf 1}_{\{ \frac{u}{3e^{n+1}} < s < \frac {3u}{e^n}\}}\ C\eps e^n\cdot \min\big\{ e^{-n}, 1 \big\}^D
\cdot \P\bigg[ \Pi_{k_u-1}B_{k_u} > \frac{\eps u}{e^{n+1}}, Y_{k_u-1}\in ds\bigg]\\
& \le \sum_{n\in \Z}
 C\eps e^n\cdot \min\big\{ e^{-n}, 1 \big\}^D
\cdot \P\bigg[ \Pi_{k_u-1}B_{k_u} > \frac{\eps u}{e^{n+1}},\ \frac{u}{3e^{n+2}} < Y_{k_u-1} < \frac{3u}{e^{n}} \bigg]\\
& \le \sum_{n\in \Z}
 C\eps e^n\cdot \min\big\{ e^{-n}, 1 \big\}^D
\cdot \P\bigg[ \Pi_{k_u-1}B_{k_u} > \frac{\eps u}{e^{n+1}},\  Y^+_{k_u-1} > \frac{u}{3e^{n+2}} \bigg]\\
& \le \sum_{n\in \Z}
 C\eps e^n\cdot \min\big\{ e^{-n}, 1 \big\}^D \frac 1{\eps^{\theta}}\ \frac{u^{-\ov \a}}{\sqrt{\log u}}\cdot e^{\a n+\d |n|}
 \le {C}{\eps^{1- \theta}}\ \cdot\frac{u^{-\ov \a}}{\sqrt{\log u}}
\end{align*}
and Proposition \eqref{eq: 6.1} follows.
}
\end{proof}

\begin{lem}\label{cor: diffomega} Assume that \eqref{H1},  \eqref{H2},  \eqref{H1.4}, \eqref{H1.6} and \eqref{H1.2} are satisfied. For any fixed (large)   $L>0$ there exist $C$ such that for every (small) $\epsilon , \eta >0$, $0<j\leq L$ and $u\geq 2$  we have
\begin{equation}
\label{eq: cdu1}\P\bigg[ \Pi_{k_u-1}B_{k_u} > \eps u ,\  \Pi_{k_u-L}Y_{j} '   \in u(1-\eta^L , 1+\eta^L) \bigg]\leq C  \eps^{-\a}\eta^L \cdot\frac{u^{-\ov \a}}{\sqrt{\log u}}.
\end{equation}
Moreover, for every $\epsilon , \eta >0$ and $u\geq 2$
\begin{equation}
\label{eq: cdu2}\P\bigg[ \Pi_{k_u-1}B_{k_u} > \eps u ,\  \Pi_{k_u-L}M_L'   \in u(1-\eta^L , 1+\eta^L) \bigg]\leq CL\eps^{-\a}\eta^L \cdot\frac{u^{-\ov \a}}{\sqrt{\log u}}
\end{equation}
and\begin{equation}\label{eq: cdu3}
\P \Big\{   \Pi_{k_u}B_{k_u+1}>\eps_0 u,      \Pi_{k_u-L} Y'_L + \Pi_{k_u}B_{k_u+1}    \in    (1-\eta^L, 1+\eta^L)u, \Big\} \leq C\eps^{-\a}\eta^L \cdot\frac{u^{-\ov \a}}{\sqrt{\log u}} .
\end{equation}
\end{lem}

\begin{proof} To prove \eqref{eq: cdu1}, we use the same argument as for the proof of  Proposition \ref{prop: 6.1}.
Let $J_{\eta^L}=u(1-\eta^L, 1+\eta^L)$. We have
\begin{equation}
\label{eq: diffomprooff}
\begin{split}
\P\big[\Pi_{k_u-L}Y_{j}' &\in J_{\eta^L},  \Pi_{k_u-1} B_{k_u}>\eps u \big]
\\ &=
\P\big[ A_1 \Pi '_{k_u-L-1} Y'_{j}\in  J_{\eta^L}, A_1 \Pi_{k_u-2}' B_{k_u}>\eps u \big]\\
%\P\big[ A_1  Y'_{j}\in  J_{\eta^L}, \Pi_{k_u-1} B_{k_u}>\eps u \big]\\
&\le \sum_{n\in\Z} \P\Big[ e^n\le A_1 < e^{n+1},       A_1 \Pi_{k_u-L-1}' Y'_{j}\in  J_{\eta^L}           \mbox{ and }\  \Pi '_{k_u-2}B_{k_u} > \frac {\eps u}{e^{n+1}}\Big]\\
&\le \sum_{n\in\Z} \int {\bf 1}_{\{ \frac{u}{2e^{n+1}} < s < \frac {2u}{e^n}\}}\ \P\Big[e^n\le A_1 < e^{n+1},  A_1\in \frac 1s J_{\eta^L} \Big]\\
&\qquad \cdot \P\Big[ \Pi '_{k_u-2}B_{k_u} > \frac{\eps u}{e^{n+1}},      \Pi_{k_u-L-1}' Y'_{j} \in ds\Big]
\end{split}
\end{equation}
For $s$ in the domain of the integral, the length of $ \frac 1s J_{\eta^L}$ is at most $C \eta^L e^n$.
As before, $\P\Big[e^n\le A_1 < e^{n+1},  A_1\in \frac 1s J_{\eta^L} \Big] \le  C\eta^L e^n\cdot \min\big\{ e^{-n}, 1 \big\}^D$. Hence the last quantity of \eqref{eq: diffomprooff} is dominated by
$$
 \sum_{n\in \Z}
 C\eta^L e^n\cdot \min\big\{ e^{-n}, 1 \big\}^D
\cdot \P\bigg[ \Pi_{k_u-2}'B_{k_u} > \frac{\eps u}{e^{n+1}}\bigg]. $$ %\Pi_{k_u-L-1}' Y'_{j}> \frac{u}{3e^{n+2}} \bigg]\le C  \eps^{-\a}\eta^L %\cdot\frac{u^{-\ov \a}}{\sqrt{\log u}} \\
Now applying Lemma \ref{lem: 33} for a fixed $\d $ we obtain
\begin{align*}
\P\big[\Pi_{k_u-L}Y_{j}'\in J_{\eta^L}, \Pi_{k_u-1} B_{k_u}>\eps u \big]&\leq
\sum_{n\in \Z}
 C\eta ^L e^n\cdot \min\big\{ e^{-n}, 1 \big\}^D  e^{\a n+\d |n|} \frac{u^{-\ov \a}}{\sqrt{\log u}}\\
& \le C\eta ^L\eps^{- \a}\ \cdot\frac{u^{-\ov \a}}{\sqrt{\log u}}
\end{align*}
and \eqref{eq: cdu1} follows.

The estimate \eqref{eq: cdu2} follows immediately from  \eqref{eq: cdu1} and the definition   $M_L'$. The proof of \eqref{eq: cdu3} is similar to \eqref{eq: cdu1}.
\end{proof}

%%%%%%%%%%%%%%%%%%%%%%%%%%%%%%%%%%%%%%%%%%%%%%%%%%%%%%%%%%%%%%%%%%%%%%%%%%%%%%%
\section{Asymptotics}
\begin{proof}[Proof of Theorem \ref{thm: main theorem}]

\par
\noindent
{\bf Step 1.}
Fix $\eps _0$ and take u such that $\eps_0> u^{-\xi '}$. Then by Proposition \ref{prop: 6.1}
\begin{equation}
\label{eq: appf}
\begin{split}
\P[\tau=k_u+1] &= \P\big[ M_{k_u} \le u, M_{k_u+1} > u  \big]\\
 &= \P\big[ M_{k_u} \le u, Y_{k_u+1} > u  \big]\\
 &= \P\big[ M_{k_u} \le u,\ Y_{k_u}\in [(1-\eps_0)u,u],\  Y_{k_u+1} > u  \big]\\
&\quad + \P\big[ M_{k_u} \le u,\ Y_{k_u} < (1-\eps_0)u,\  Y_{k_u+1} > u  \big]\\
 &=  O\bigg(\frac{\eps_0^{1-\sigma}}{\sqrt{\log u}} u^{-\ov \a}\bigg) + \P\big[ M_{k_u} \le u,\ Y_{k_u} < (1-\eps_0)u,\  Y_{k_u+1} > u  \big].
\end{split}
\end{equation}
Thus is is sufficient to prove that
\begin{equation}\label{eq: 8.0}
\lim _{u\to \8 }u^{\ov \a}\sqrt{\log u}\; \P\big[ M_{k_u} \le u,\ Y_{k_u} < (1-\eps_0)u,\  Y_{k_u+1} > u  \big]\quad \mbox{exists}
%\sim \frac{C}{\sqrt{\log u}} u^{-\ov \a}
\end{equation} for some fixed arbitrary small $\eps_0$. Indeed, having proved \eqref{eq: 8.0} we first let $u\to \infty $ and then $\eps _0\to 0$.

Similar arguments as in the proof of Lemma \ref{lem: 333} show that there are constants: large $L_0$ and $\eta,\eta_1<1$ and $C$ possibly depending on $\epsilon_0$ such that for any $L>L_0$
we have
\begin{equation}
\label{eq: wtss}
 \P\Big[ M_{k_u-L}\ge  \eta^L u,\;\Pi_{k_u} B_{k_u+1} >\eps_0 u\Big] \le \eta_1^L\ \frac {C}{\sqrt{\log u}}\ u^{-\ov \a}
\end{equation}
and
\begin{equation}
\label{eq: wtss'}
 \P\Big[ |Y_{k_u-L}|\ge  \eta^L u,\;\Pi_{k_u} B_{k_u+1} >\eps_0 u\Big] \le \eta_1^L\ \frac {C}{\sqrt{\log u}}\ u^{-\ov \a}
\end{equation}
Therefore, choosing large (but fixed) $L$, to prove the main result, it is enough to show that
\begin{equation}
\label{wtsss}
\lim_{u\to\8} u^{\ov \a} \sqrt{\log u}\, \P[\Omega] \quad \mbox{exists},
\end{equation}
where
$$
\Omega = \Big\{ M_{k_u-L}<\eta^L u,\ |Y_{k_u-L}|<\eta^L u,\  M_{k_u}< u,\  Y_{k_u}<(1-\eps_0)u,\ Y_{k_u+1}>u \Big\}.
$$
As before, we conclude letting first $u\to \infty $ then $L\to \infty $.
The limit in \eqref{wtsss}, if it exists,  has to be positive. Indeed, taking $L$ large, we can make the upper bounds in \eqref{eq: wtss} and \eqref{eq: wtss'}
smaller than the lower bound in
 Proposition \ref{lem: lower estimates large n}.
%, it follows from  \eqref{eq: appf}
%\eqref{eq: wtss} and \eqref{eq: wtss'} that  the limit, if exist,  has to be %positive.

\noindent
{\bf Step 2.} To prove \eqref{wtsss} we modify further the set $\Omega$ and as we will see, it is sufficient to replace $\Omega$ by a set $\Omega_2$ defined below.
By the definition of $M_{k_u}$
$$
M_{k_u} = \max_{j=1,\ldots,k_u}\{Y_j\} = \max\Big\{ M_{k_u-L}, Y_{k_u-L} + \Pi_{k_u-L}M'_L \Big\},
$$
where $M'_L = \max_{j=1,\ldots,L}\sum_{i=k_u-L}^jA_{k_n-L+1}\cdots A_{i-1}B_i$. Notice that $M'_L$ has the same law as $M_L$.

Define the sets
\begin{align*}
\Omega_1 &= \Big\{ M_{k_u-L}<\eta^L u, \Pi_{k_u-L} M'_L<(1-\eta^L)u,\
|Y_{k_u-L}|<\eta^L u, \Pi_{k_u-L} Y'_L<(1-\eps_0-\eta^L)u,\ \\ &\qquad  \Pi_{k_u-L} Y'_L + \Pi_{k_u}B_{k_u+1}>(1+\eta^L)u
\Big\}\\
\Omega_2 &= \Big\{ \Pi_{k_u-L} M'_L\!<\!(1+\eta^L)u,
 \Pi_{k_u-L} Y'_L\!<\!(1\!-\!\eps_0\!+\!\eta^L)u, \Pi_{k_u-L} Y'_L + \Pi_{k_u}B_{k_u+1}\!>\!(1-\eta^L)u
\Big\}\\
\end{align*}
Then
$$
\Omega_1 \subset \Omega \subset \Omega_2.
$$
It is convenient to modify slightly $\Omega_1$ and consider
\begin{align*}
\Omega'_1 &= \Omega_1\cup\Omega''_1\\& =
\Big\{ \Pi_{k_u-L} M'_L<(1-\eta^L)u,\ \Pi_{k_u-L} Y'_L<(1-\eps_0-\eta^L)u,\\ &\qquad  \Pi_{k_u-L} Y'_L + \Pi_{k_u}B_{k_u+1}>(1+\eta^L)u
\Big\}
\end{align*}
for
\begin{align*}
\Omega''_1 &= \Big\{ M_{k_u-L}\ge \eta^L u,\ \Pi_{k_u-L} M'_L<(1-\eta^L)u,\ \Pi_{k_u-L} Y'_L<(1-\eps_0-\eta^L)u,\ \\ &\qquad  \Pi_{k_u-L} Y'_L + \Pi_{k_u}B_{k_u+1}>(1+\eta^L)u
\Big\}
\cup \Big\{|Y_{k_u-L}|\ge \eta^L u,\ \Pi_{k_u-L} M'_L<(1-\eta^L)u,\  \\ &\qquad \Pi_{k_u-L} Y'_L<(1-\eps_0-\eta^L)u,\  \Pi_{k_u-L} Y'_L + \Pi_{k_u}B_{k_u+1}>(1+\eta^L)u
\Big\}
\end{align*}
Notice that by \eqref{eq: wtss} and \eqref{eq: wtss'}
\begin{align*}
\P\big[ \Omega''_1 \big]&\le
 \P\Big[ M_{k_u-L}> \eta^L u,\ \Pi_{k_u}B_{k_u+1}>\eps_0 u \Big] +
  \P\Big[ |Y_{k_u-L}| > \eta^L u,\ \Pi_{k_u}B_{k_u+1}>\eps_0 u \Big]\\
 &\le \eta_1^L\; \frac{C_{\eps_0}}{\sqrt{\log u}} \ u^{-\ov \a}.
\end{align*}
We have
$$
\P\big[ \Omega'_1 \big] - \P\big[ \Omega''_1 \big]\le \P\big[ \Omega \big]\le \P\big[ \Omega_2 \big]
$$
We claim  that $\P[\Omega_2 \setminus \Omega_1']\leq  C\eps_0^{-\a}\eta^L \cdot\frac{u^{-\ov \a}}{\sqrt{\log u}}$. We have
\begin{eqnarray*}
\P[\Omega_2 \setminus \Omega_1']
&\leq&\P \Big\{ \Pi_{k_u-L} M'_L \in    (1-\eta^L, 1+\eta^L)u,  \Pi_{k_u}B_{k_u+1}>\eps_0 u\Big\}\\
&+&\P \Big\{ \Pi_{k_u-L} Y'_L \in    (1-\eps_0 -\eta^L, 1-\eps_0 +\eta^L)u,  \Pi_{k_u}B_{k_u+1}>\eps_0 u\Big\}\\
&+&\P \Big\{  \Pi_{k_u-L} Y'_L + \Pi_{k_u}B_{k_u+1}    \in    (1-\eta^L, 1+\eta^L)u,  \Pi_{k_u}B_{k_u+1}>\eps_0 u\Big\}.
\end{eqnarray*}
and the claim follows from Lemma  \ref{cor: diffomega} applied to each summand. Now it suffices to prove that
 $$\lim_{u\to\8} u^{\ov \a} \sqrt{\log u}\, \P[\Omega_2] \quad  \mbox{exists.}$$
\noindent
{\bf Step 3.}
To proceed we write
%let us first estimate the following expression
\begin{multline*}
P(u,\eps,\d,\g) \\ = \P \Big[   \Pi_{k_u-L} M'_L<(1+\eps)u,\  \Pi_{k_u-L} Y'_L<(1+\d)u,\  \Pi_{k_u-L} Y'_L + \Pi_{k_u}B_{k_u+1}>(1+\g)u \Big],
\end{multline*}
where $\eps , \d , \g $ may have arbitrary signs.
Notice that
\begin{equation}
%\P(\Omega'_1) &=& P(u,-\eta^L, -\eps_0-\eta^L, \eta^L),\\
\P(\Omega_2) = P(u,\eta^L, -\eps_0+\eta^L, -\eta^L).
\end{equation}
We are going to prove that
$$
P(u,\eps,\d,\g)=C_L\frac{u^{-\ov \a}}{\sqrt{\log u}}(1+o(1))$$ for some constant $C_L$.
In fact, it is sufficient to show that
\begin{equation}
\label{eq: dss}
\wt P(u,\eps,\d,\g)=C_L\frac{u^{-\ov \a}}{\sqrt{\log u}}(1+o(1)),
\end{equation}
where
\begin{align*}
\wt P(u,&\eps,\d,\g) = \P \bigg[   \Pi_{k_u-L} M'_L<(1+\eps)u,\  \Pi_{k_u-L} Y'_L<(1+\d)u,\ \\ &
\Pi_{k_u-L} Y'_{L+1}\!=\!\Pi_{k_u-L} Y'_L + \Pi_{k_u}B_{k_u+1}>(1+\g)u,\ \mbox{ and } ue^{-(\log u)^{\frac 14}}\!<\! \Pi_{k_u-L} \!<\! u e^{(\log u)^{\frac 14}} \bigg].
\end{align*}
Indeed
\begin{eqnarray*}
P(u,\eps,\d,\g)- \wt P(u,\eps,\d,\g) &\le & \P \Big[ \Pi_{k_u-L}> ue^{(\log u)^{\frac 14}}  \Big]\\
&& + \P \Big[ \Pi_{k_u-L}< ue^{-(\log u)^{\frac 14}} \mbox{ and } \Pi_{k_u}B_{k_u+1} > (\g-\d) u \Big].
\end{eqnarray*}
Applying the Chebychev inequality with $\a $ to the first term we have
$$
\P \Big[ \Pi_{k_u-L}> ue^{(\log u)^{\frac 14}}  \Big] \le e^{-\a (\log u)^{\frac 14}} u^{-\a} \lambda(\a)^{k_u-L} = o\bigg(\frac{u^{-\ov \a}}{\sqrt{\log u}}\bigg).
$$ For the second term we choose $\b >\a $ and we write
\begin{align*}
\P \Big[ & \Pi_{k_u-L} < ue^{-(\log u)^{\frac 14}} \mbox{ and } \Pi_{k_u}B_{k_u+1} > (\g-\d) u \Big]\\
&\le \sum_{m\ge 0 }
\P \Big[ ue^{-(\log u)^{\frac 14}} e^{-(m+1)}\le   \Pi_{k_u-L}\!<\! ue^{-(\log u)^{\frac 14}} e^{-m}\Big]\cdot
\P\Big[  \Pi_L B > (\g - \d) e^{(\log u)^{\frac 14}} e^{m}  \Big]\\
&\le \sum_{m\ge 0} \frac{C e ^{\a (\log u)^{\frac 14}} e^{\a m} }{u^\a}\, \lambda(\a)^{k_u-L}\, \frac{\lambda(\b)^L \E |B|^\b}{e^{\b (\log u)^{\frac 14} }  e^{\b m} } \\ &= o\bigg( \frac{u^{-\ov \a}}{\sqrt{\log u}} \bigg),
\end{align*}
applying Chebychev with $\a $ and $\b $ respectively. The our proof is reduced to \eqref{eq: dss}.

\noindent
{\bf Step 4.} Finally,
notice that  $\wt P(u,\eps,\d,\g)$ is the probability of a set on which
$$ ue^{-(\log u)^{\frac 14}}<\Pi_{k_u-L} < u e^{(\log u)^{\frac 14}}$$
%\frac{1}{2}e^{-(\log u)^{\frac 14}}<
%Y'_L, Y'_{L+1}, M'_L<2e^{(\log u)^{\frac 14}}$$
%since we may assume that $\eps , \d , \g \in (\frac 12 , 1)$.

Therefore, we may
apply the Petrov theorem and we have
\begin{align*}
\wt P(u,&\eps,\d,\g) \\
&
=\int  %_{\{ % e^{-(\log u)^{\frac 14}}<% s_1, s_2, s_3 < 2 e^{(\log u)^{\frac 14}}   \}}
\P \Big[u \max \Big (\frac{(1+\g)}{s_3}, \e^{-(\log u)^{\frac 14}}\Big ) <  \Pi_{k_u-L} <u \min \Big ({\frac{(1+\eps)}{s_1}, \frac{(1+\d)}{s_2},{\e^{(\log u)^{\frac 14}}}} \Big ) \Big]\\
& =
\P \Big[  M'_L  \in (s_1, s_1+ds_1),\  Y'_L \in  (s_2, s_2+ds_2)  ,
Y'_{L+1} \in  (s_3, s_3+ds_3) \Big]
\\
&
=\frac{C\,u^{-\ov \a}}{\sqrt{\log u}}(1+o(1))\E \bigg [\bigg (\Big (\frac{Y'_{L+1}}{1+\gamma}\Big )^\a -\max \bigg(\Big (\frac{M'_L}{(1+\eps)}\Big )^\a,
\Big (\frac{Y'_L}{(1+\d)}\Big )^\a \bigg ) \bigg )_{+} \bigg ]
\end{align*}
 %\frac{(1+\d)^\a}{s_2^\a}} )_{+}$$

%\int_{\{  e^{-(\log u)^{\frac 14}}< s_1, s_2, s_3 <  e^{(\log u)^{\frac 14}}   %\}}
%((\frac{s_3}{1+\gamma})^\a -\min {\frac{(1+\eps)^\a}{s_1^\a}, %\frac{(1+\d)^\a}{s_2^\a}} )_{+}$$
%$$
%\P \Big[  M'_L  \in (s_1, s_1+ds_1),\  Y'_L \in  (s_2, s_2+ds_2)  ,
%Y'_{L+1} \in  (s_3, s_3+ds_3) \Big].
%$$
The last integral is finite by moment assumption \eqref{H2} and the conclusion  follows.
\end{proof}

%%%%%%%%%%%%%%%%%%%%%%%%%%%%%%%%%%%%%%%%%%%%%%%%%%%%%%%%%%%%%%%%%%%%%%%%%%%%%%%
\section{Proof of Theorem \ref{mthm2}}\label{sec: thm2}
First for $\b <\a +\xi $ we will need the following statement which is, in fact, Lemma \ref{lem: 5.9} with $r=D_1$.
\begin{lem}
\label{lem: s} Under assumptions of Theorem \ref{mthm2}, there are constants $ D_1, C_1, C$ such that for $n=\frac{\log u}{\Lambda'(\b)}$ and every $\g >0$
$$
\P\big[ \g u < \Pi_n < D_1 \g u, Y_n < C_1 \g u
\big] \ge \frac{C \lambda(\b)^n}{u^\b \g ^{\b}\sqrt{\log u}}
$$
%We can assume, that $(1+C_1)\gamma$ is small.
\end{lem}
%This  is  the statement of Lemma \ref{lem: 5.9} with fixed $r=2$.

%\cite{BDMZ}.{\color{red} BCDZ?} Indeed there $B$ is assumed to be identically equal to 1, but the same proof (with minor modifications) works also in our case.% producing some constants. Now we easily can make $\gamma$ small using \ref{ H2.4}.
\begin{lem}
\label{lem: ss} Under assumptions of Theorem \ref{mthm2}, there are constants $D_2, C_2, C_3, C, m_0$ be such that for any $m>m_0$ and $\eps = e^{m\Lambda'(1)}<1$
$$
\P\big[ \eps \le \Pi_m \le D_2 \eps, C_3\le  Y_m \le C_2 \big]
\ge \frac{C \lambda(1)^m}{\eps \sqrt m}.
$$
\end{lem}
\begin{proof}

\noindent
{\bf Step 1.}
We change the probability space and consider the probability measure $\lambda(1)^{-1}a\mu (da,db)$. Denote by $\P_1$ the corresponding probability measure on the space of trajectories and by $\E_1$ the corresponding expected value. Then, for $S_m  =\log \Pi_m$,
\begin{eqnarray*}
\P[\eps\le \Pi_m \le D_2 \eps] &\ge& \frac{\lambda(1)^m}{D_2 \eps} \E \bigg[ {\bf 1}_{\{ \eps \le \Pi_m \le D_2 \eps \}} \frac{\Pi_m}{\lambda(1)^m}\bigg] \\
&=& \frac{\lambda(1)^m}{D_2\eps} \P_1\big[ 0\le S_m - m\Lambda'(1) \le \log D_2
\big].
\end{eqnarray*}
Since $S_m$ is a sum of iid random variables, $\E_1 S_m = m\Lambda'(1)$ and $\E_1 S_1^2<\8$ (because of (H2)), by the local limit theorem
$$\P[\eps\le \Pi_m \le D_2 \eps] \ge \frac{C \lambda(1)^m}{\eps \sqrt m}.$$
{\bf Step 2.} Denote $\Pi'_{m-j-1} = A_{j+2}\ldots A_m$. We have
\begin{eqnarray*}
\P\big[ \eps \le \Pi_m \le D_2\eps, Y_m \ge C_2 \big]&\le& \sum _{j=1}^{m-1} \P\bigg[ \eps < \Pi_m \le D_2 \eps, \Pi_jB_{j+1} > \frac{C_2}{2j^2}
\bigg]\\
&=& \sum _{j=1}^{m-1} \int \P\bigg[ \eps < \Pi_j a \Pi'_{m-j-1}< D_2 \eps,\, \Pi_j b > \frac{C_2}{2j^2}
\bigg]\mu(da,\,db)
%&=& \sum_j P_j
\end{eqnarray*}
and for every $j$ we consider
$$P_j=\int \P\bigg[ \eps < \Pi_j a \Pi'_{m-j-1}< D_2 \eps,\, \Pi_j b > \frac{C_2}{2j^2}
\bigg]\mu(da,\,db).$$
{\bf Step 2a.} Since $\alpha _{min}>1$, $\L (1), \L'(1)<0$ and so we can choose  $r<1$ and $\beta$ such that $r\Lambda(\b)<\Lambda(1)-\Lambda'(1)$ and $\L(\b)<0$. Then, by the Chebyshev inequality,
\begin{eqnarray*}
\sum_{j> rm} P_j &\le&
\sum_{j> rm} \int \P\bigg[ \Pi_j b > \frac{C_2}{2j^2} \bigg]\mu(da,db)\\
&\le&\sum_{j> rm} \frac{(2j^2)^\b}{C_2^\b} \lambda(\b)^j \E B^\b \le  C  e^{rm(\Lambda(\b)+\d)}\\
&=& o\bigg(\frac{\lambda(1)^m}{\eps \sqrt m}\bigg)
\end{eqnarray*}
for a positive $\d $ such that $r(\L (\b )+\d) <\L (1)-\L '(1)$.

{\bf Step 2b.} For $j\le rm$ we write
\begin{eqnarray*}
\sum_{j\le rm} P_j &\le&\sum_{j\le rm}\sum_{k\ge 0} \int \P\bigg[ \eps < \Pi_j a \Pi'_{m-j-1}< D_2 \eps,\, \frac{C_2 e^k}{2j^2 b}<  \Pi_j  <  \frac{C_2 e^{k+1}}{2j^2 b}
\bigg] \mu(da,\, db)\\
&\le&\sum_{j\le rm}\sum_{k\ge 0}
\int \P\bigg[ \frac{2\eps j^2 b}{C_2 a e^{k+1}} <  \Pi_{m-j-1}<  \frac{2D_2 \eps j^2 b}{C_2 a e^{k}}\bigg]\cdot\P\bigg[
  \Pi_j  >  \frac{C_2 e^{k}}{2j^2 b}
\bigg] \mu(da,\, db)
\end{eqnarray*}
To proceed further we recall the  Berry-Essen theorem (see e.g. \cite{Petrovbook}) that
for an i.i.d. sequence $\{X_j\}$ with variance $\sigma^2$ and finite third moment,
gives
\begin{equation}
\label{eq: berry-essen}
\sup_x \bigg| \P\bigg[\frac{\sum_{j=1}^m X_j - m\E X_1 }{\s \sqrt m} < x\bigg] - \Phi(x) \bigg| \le \frac{\ov C\ \E|X_1-\E X_1|^3}{\s^3}\cdot \frac 1{\sqrt m}=\frac{\g }{\sqrt m}
\end{equation}
where $\Phi$ denotes the normal distribution and $\ov C$ is a universal constant.
Hence, changing again the probability space, we have
\begin{align*}
\P\bigg[  \frac{2\eps j^2 b}{C_2 a e^{k+1}} & <  \Pi_{m-j-1}<  \frac{2D_2 \eps j^2 b}{C_2 a e^{k}}\bigg]\\
&\le
\frac{\lambda(1)^{m-j-1} C_2 a e^{k+1}}{2\eps j^2 b}\E\bigg[ {\bf 1}_{\big\{\frac{2\eps j^2 b}{C_2 a e^{k+1}} <  \Pi_{m-j-1}<  \frac{2D_2 \eps j^2 b}{C_2 a e^{k}}\big\}} \frac{\Pi_{m-j-1}}{\lambda(1)^{m-j-1}} \bigg]\\
&= \frac{\lambda(1)^{m-j-1} C_2 a e^{k+1}}{2\eps j^2 b} \P_1\Big[
m\Lambda'(1)+\log(2j^2/C_2) + \log(b/a) - (k+1) < S_{m-j-1}\\
& \quad < m\Lambda'(1)+\log(2D_2j^2/C_2) + \log(b/a) - k \Big]\\
&= \frac{\lambda(1)^{m-j-1} C_2 a e^{k+1}}{2\eps j^2 b} \P_1\bigg[\frac{
(j+1)\Lambda'(1)+\log(2j^2/C_2) + \log(b/a) - (k+1)}{\sqrt{m-j-1}}\\
& \quad < \frac{S_{m-j-1} \!-\!(m\!-\!j\!-\!1)\Lambda'(1) }{\sqrt{m-j-1}} <
\frac{(j+1)\Lambda'(1)\!+\!\log(2D_2j^2/C_2)\! +\! \log(b/a) \!-\! k}{\sqrt{m-j-1}} \bigg]\\
&\le \frac{\g C_2 \lambda(1)^{m-j-1}  a e^{k+1}}{2\eps j^2 b\sqrt{m-j-1}}.
\end{align*}
\eqref{eq: berry-essen} has been used in the last inequality. For $\P\big[
  \Pi_j  >  \frac{C_2 e^{k}}{2j^2 b} \big]$ we use the Chebychev inequality with $\a $ and so
\begin{eqnarray*}
\sum_{j\le rm} P_j &\le& C \sum_{j\le rm}\sum_{k\ge 0} \frac{C_2 \lambda(1)^{m-j-1}e^k}{\eps j^2\sqrt{m-j-1}}\frac{ j^{2\a} \lambda(\a)^j}{C_2^\a e^{\a k}}\int ab^{\a-1}\mu(da,\, db)\\
&\le& C C_2^{1-\a} \frac{\lambda(1)^m}{\eps \sqrt m} \sum_{j\le rm } \bigg(\frac{\lambda(\a)}{\lambda(1)}\bigg)^j j^{2\a}\\
&\le& \frac{C}{C_2^{\a-1}} \frac{\lambda(1)^m}{\eps\sqrt m}
\end{eqnarray*}
{\bf Step 3.}
Combining first two steps and taking large $C_2$ we obtain
$$
\P\big[ \eps \le \Pi_m \le D_2 \eps,  Y_m \le C_2 \big]
\ge \frac{C \lambda(1)^m}{\eps \sqrt m}.
$$ Take parameters $a_2>1$, $ b_1<b_2$, $\eta>0$ such that
$$ \P\big[ A\in (1,a_2),\,B\in (b_1,b_2)
\big] \ge \eta >0
$$
Then, for $(A_0, B_0)$ independent of $\Pi _m$ and $Y_m$, we have
\begin{eqnarray*}
\eta \frac{C \lambda(1)^m}{\eps \sqrt m}
&\le& \P\big[ A_0\in (1,a_2),\,B_0\in (b_1,b_2)
\big] \cdot  \P\big[ \eps \le \Pi_m \le D_2 \eps,  Y_m \le C_2 \big]\\
&=&
\P\big[ A_0\in (1,a_2),\,B_0\in (b_1,b_2)\ \mbox{ and } \
 \eps \le \Pi_m \le D_2 \eps,  Y_m \le C_2 \big]\\
&\le&
\P\big[
 \eps \le \Pi_{m+1} \le a_2 D_2 \eps, b_1 \le  Y_m \le b_2 + a_2 C_2 \big].
\end{eqnarray*}
Thus the Lemma follows for $C_3=b_1$ and for $D_2, C_2$ replaced by $a_2 D_2$,  $b_2+a_2 C_2$.
\end{proof}

\begin{proof}[Proof of Theorem \ref{mthm2}]
Let  $\b>\a$ be very close to $\a$ and define $p$ by the relation
\begin{equation}
\label{eq: sss}
\Lambda'(\b) = \frac{\Lambda'(\a)}{p}.
\end{equation}
Since $\L''(\b)>0$, we have $p<1$.
The precise value of $\b$ will be chosen later on. Take $n=pk_u = \frac{\log u}{\Lambda'(\b)}$.  Let $q = 1-p$,  $m=qk_u = k_u-n$. We write
$$
Y_{k_u} = Y_n  + \Pi_{n-1} A_n Y'_m = Y_{n-1} + \Pi_{n-1}B_n + \Pi_{n-1}A_n Y'_m.
$$ for $Y'_m = \sum_{i=n+1}^{k_u} A_{n+1}\ldots A_{i-1} B_i$. Then $Y'_m$ has the same law as $Y_m$. We denote also $\Pi'_j = A_{n+1}\ldots A_{n+j}$, $B'_m = B_{k_u}$ and we consider the set

$$
\Omega = \Big\{ \g u < \Pi_{n-1} < D_1 \g u, Y_{n-1} \le C_1 \g u, C_3 \le Y'_m \le C_2, \eps<\Pi'_{m-1}< D_2 \eps,
b_1 <  B_m' < b_2
\Big\},
$$
where $\g,D_1, $ are parameters given in Lemma \ref{lem: s}, $C_2,C_3,D_2$ are described in Lemma \ref{lem: ss}
%(we assume here additionally $2D_1 C_2 > C_3$, but we can choose $C_3$ as small as we wish),
$\eps = e^{m\Lambda'(1)}$, $b_1$ and $b_2$ are chosen as in \eqref{H2.3}.

By Lemmas \ref{lem: s} and \ref{lem: ss}
\begin{equation}
\label{eq: ss}
\P(\Omega) \ge \frac{C \lambda(\b)^n}{u^\b}\, \frac 1{\sqrt n}\, \frac{\lambda(1)^m}{\eps \sqrt m}.
\end{equation}
We write
\begin{equation}
\begin{split}
 \label{eq: s}
  \P\big[ & Y_{k_u-1}  \le u, Y_{k_u}>u \big]\\
&= \P\Big[ Y_{n-1} + \Pi_{n-1} B_n + \Pi_{n-1} A_n Y'_m - \Pi_{n-1} A_n \Pi'_{m-1} B_m \le u,\\
&\quad  Y_{n-1} + \Pi_{n-1} B_n + \Pi_{n-1} A_n Y'_m > u
\Big]\\
&\ge \P\bigg[\bigg\{ \frac{u-Y_{n-1} -\Pi_{n-1} B_n  }{\Pi_{n-1} Y'_m} < A_n \le
\frac{u-Y_{n-1} -\Pi_{n-1} B_n  }{\Pi_{n-1} (Y'_m - \Pi'_{m-1} B'_m)}, b_1\leq B_n \leq b_2
\bigg\} \cap \Omega
\bigg]
\end{split}
\end{equation}
Notice that on the set $\Omega$, for $b_1\le B_n \le b_2$ we have
\begin{eqnarray*}
\frac{u-Y_{n-1} -\Pi_{n-1} B_n  }{\Pi_{n-1} (Y'_m - \Pi'_{m-1} B'_m)} &<&
\frac{u}{\g u C_3/2} = \frac {2}{\g C_3} =: C_+\\
\frac{u-Y_{n-1} -\Pi_{n-1} B_n  }{\Pi_{n-1} Y'_m } &>& \frac{u - C_1\g u - \g u b_1}{D_1 \g u C_2}
= \frac{1- C_1 \g - \g b_1}{D_1 \g C_2} =: C_-.
\end{eqnarray*}
Here we have used $B_m' \Pi_{m-1}' \leq D_2 b_2\eps\leq \frac{C_3}{2}$ for large $m$.
Moreover on $\Omega$
\begin{align*}
\frac{u-Y_{n-1} -\Pi_{n-1} B_n  }{\Pi_{n-1} (Y'_m - \Pi'_{m-1} B'_m)} - &
\frac{u-Y_{n-1} -\Pi_{n-1} B_n  }{\Pi_{n-1} Y'_m }\\ &\ge
\frac{u-Y_{n-1} -\Pi_{n-1} B_n  }{\Pi_{n-1} } \cdot
\frac{\Pi'_{m-1} B'_m }{Y'_m (Y'_m - \Pi'_{m-1} B'_m)}\\
&\ge \frac{1-C_1 \g - \g b_1}{D_1 \g} \cdot \frac{\eps b_1}{ C_2^2} = d_1\eps.
\end{align*}
and we have $C_{-}>0,d_1>0$ if we take $(C_1+b_1)\gamma\leq \frac12$.

The set $\Omega$ is independent on $A_n, B_n$. By \eqref{eq: s}  for fixed value of $b=B_n$, we can take any  $A_n$ in \eqref{eq: s} from
some interval $I_{d_1 \eps, b}\subset (C_-,C_+)$ (depending on $b$) of length at least $d_1 \eps$.
In view of \ref{H2.3}
%there is a nonvanishing continuous function $f_A$ such that
%$$
%\P[A\in (a,a+ da), b_1 < B_n < b_2, B_n\in (b, b+db)] \ge f_A(a)da d\nu(b).
$d_2 = \inf_{C_- < a < C_+ } g_A(a)$ is strictly positive. Then, by \eqref{eq: ss}
$$
\P\big[ Y_{k_u-1} < u, Y_{k_u} > u \big] \ge \inf_{I_{d_1 \eps, b}\subset ( C_-, C_+)}\P\big[ A\in I_{d_1\eps} \big] \P[\Omega]
\ge Cd_1d_2 \cdot \frac{\lambda(\b)^n}{u^\b} \, \frac 1{\sqrt n} \, \frac{\lambda(1)^m}{\sqrt m}
$$
Since $\sqrt{nm}$ is of order $\log u$, to finish the proof we have to justify that
$$
 \frac{\lambda(\b)^n \lambda(1)^m }{u^\b} \ge \frac{\lambda(\a)^{k_u}}{u^\a} \cdot u^\d
$$ for some $\d$. In other words we want to show
$$
\bigg(\frac{\lambda(\b)}{\lambda(\a)}\bigg)^n
\bigg(\frac{\lambda(1)}{\lambda(\a)}\bigg)^m \ge u^{\b-\a} u^\d.
$$
Choose $p$ in \eqref{eq: sss} such that $\b-\a < \eta$ for $\eta = \frac{\log \mu}{\Lambda'(\a)}$ and $\mu = \frac{\lambda(1)}{\lambda(\a)}>1$. Then, by the Taylor expansion of $\Lambda$, since $\L''(\b)>0$, $\Lambda (\b)-\Lambda (\a )\geq \Lambda '(\a )(\b -\a )$ and so
\begin{eqnarray*}
\bigg(\frac{\lambda(\b)}{\lambda(\a)}\bigg)^n
\bigg(\frac{\lambda(1)}{\lambda(\a)}\bigg)^m &=&e^{n(\Lambda(\b)- \Lambda(\a))}\mu^m\\ &\geq& e^{n(\b-\a)\Lambda'(\a)}\mu^m
\end{eqnarray*}
But $n(\b-\a)\Lambda'(\a)=(\b -\a )p\log u$ and $\mu ^m=
e^{\log \mu\cdot \frac{q\log u}{\Lambda'(\a)}}=e^{q\eta \log \mu}$. Hence
\begin{eqnarray*}
\bigg(\frac{\lambda(\b)}{\lambda(\a)}\bigg)^n
\bigg(\frac{\lambda(1)}{\lambda(\a)}\bigg)^m
&\ge & u^{p(\b-\a)} u^{\eta q}\\
&=& u^{\b-\a} u^{(\eta - (\b-\a))q},
\end{eqnarray*}
which completes proof of the Theorem.
\end{proof}
\begin{rem}
\label{rem: 8.1}
The assumption on $g_A(a)$ in \eqref{H2.3} can be weakened to  $g_A(a)\geq c >0$ on some interval $(a_1, a_2)$ with $a_1 >1$.
The proof is similar. It requires only a more careful definition of $\Omega$.
If   $g_A(a)da$ contains a nontrivial absolutely
continous part, by Steinhaus theorem, its convolution power has to satisfy the condition above. We leave the details for the reader.
\end{rem}

\bibliographystyle{alpha}
\newcommand{\etalchar}[1]{$^{#1}$}

\end{document}